\documentclass[a4paper]{amsart}
\usepackage{amssymb}
\usepackage{amsmath}
\usepackage{xypic}
\usepackage{esint}

\newtheorem{thm}{Theorem}[section]
\newtheorem*{thm*}{Theorem}
\newtheorem{thmintro}{Theorem}

\newtheorem{cor}[thm]{Corollary}
\newtheorem*{cor*}{Corollary}
\newtheorem{lemma}[thm]{Lemma}
\newtheorem{prop}[thm]{Proposition}
\theoremstyle{definition}
\newtheorem{defn}[thm]{Definition}
\newtheorem{ex}[thm]{Example}

\theoremstyle{remark}
\newtheorem{rem}[thm]{Remark}

\newcommand{\ZZ}{{\mathbb  Z}}
\newcommand{\RR}{{\mathbb  R}}
\newcommand{\CC}{{\mathbb  C}}

\newcommand{\caD}{{\mathcal D}}

\newcommand{\cS}{{\mathcal S}}

\newcommand{\E}{{\mathcal E}}

\newcommand{\F}{{\mathcal F}}

\newcommand{\so}{{{\mathfrak{so}}} }

\newcommand{\hor}{{\rm hor}\, }

\newcommand{\whF}{{\widehat{\F}}}

\newcommand{\mfg}{\mathfrak{g}}
\newcommand{\mfh}{\mathfrak{h}}

\newcommand{\mfa}{\mathfrak{a}}

\newcommand{\mfk}{\mathfrak{k}}
\newcommand{\mft}{\mathfrak{t}}

\newcommand{\Ad}{\operatorname{Ad}}

\newcommand{\codim}{\operatorname{codim}}

\newcommand{\pr}{\operatorname{pr}}

\newcommand{\id}{\operatorname{id}}

\newcommand{\SO}{\operatorname{SO}}

\newcommand{\Isom}{\operatorname{Isom}}
\newcommand{\Lie}{\operatorname{Lie}}

\newcommand{\Vol}{\operatorname{Vol}}

\newcommand{\supp}{\operatorname{supp}}
\newcommand{\cod}{\operatorname{cod}}
\newcommand{\tr}{\operatorname{tr}}

\numberwithin{equation}{section}

\begin{document}

\title[Localization of Chern-Simons type invariants
]{Localization of Chern-Simons type invariants of Riemannian foliations
}

\author{Oliver Goertsches}
\address{Oliver Goertsches, Mathematisches Institut der Universit\"at M\"unchen, Theresienstr.\ 39, 80333 M\"unchen, Germany}
\email{goertsches@math.lmu.de}
\author{Hiraku Nozawa}
\address{Hiraku Nozawa, Department of Mathematical Sciences,
Colleges of Science and Engineering, Ritsumeikan University,
1-1-1 Nojihigashi, Kusatsu, Shiga, 525-8577, Japan}
\email{hnozawa@fc.ritsumei.ac.jp}
\thanks{The second author is partially supported by Research Fellowship of the Canon Foundation in Europe, the EPDI/JSPS/IH\'{E}S Fellowship, the Spanish MICINN grant MTM2011-25656 and JSPS KAKENHI Grant Number 26800047.}
\author{Dirk T\"oben}
\address{Dirk T\"oben, Universidade Federal de S\~ao Carlos, Departamento de Matem\'atica, Rod.\ Washington Lu\'is, Km 235 - C.P.\ 676 - 13565-905 S\~ao Carlos, SP - Brazil}
\email{dirktoben@dm.ufscar.br}

\date{}


\begin{abstract}
We prove an Atiyah-Bott-Berline-Vergne type localization formula for Killing foliations in the context of equivariant basic cohomology.
As an application, we localize some Chern-Simons type invariants, for example the volume of Sasakian manifolds and secondary characteristic classes of Riemannian foliations, to the union of closed leaves. Various examples are given to illustrate our method.
\end{abstract}

\maketitle

\tableofcontents

\addtocontents{toc}{\protect\setcounter{tocdepth}{1}}
\section{Introduction}\label{sec:intro}

\subsection{Background and motivation}

Given a torus action on an oriented compact manifold, the Atiyah-Bott-Berline-Vergne localization formula \cite{AtiyahBott,BerlineVergne} allows one to calculate the integral of certain top cohomology classes like characteristic classes, as a related integral over the fixed point set. This is a generalization of an old result by Bott \cite{Bott1}: Given an oriented compact Riemannian manifold $M$ with a Killing vector field $X$, the Pontryagin numbers of $M$ can be computed in terms of the zero set of $X$. If there are no zeroes, these numbers are zero. But even in this case $M$ can have nontrivial Chern-Simons invariants. As we will see, a natural example is the volume of a Sasakian, or more generally a $K$-contact manifold.

The purpose of this paper is to prove a foliated version of the ABBV formula; for certain Riemannian foliations we want to localize Chern-Simons type invariants to the union of closed leaves. We will apply the formula to compute the volume of $K$-contact manifolds and some secondary characteristic numbers of Riemannian foliations.

In our ABBV type formula, we decompose Chern-Simons type invariants of foliations into leafwise cohomology classes (the tangential part) and basic cohomology classes (the transverse part). This idea is similar to those of Duminy~\cite{Duminy}, Cantwell-Conlon~\cite{CantwellConlon}, Heitsch-Hurder~\cite{HeitschHurder}, Hurder-Katok~\cite{HurderKatok} and Asuke \cite{Asuke2} to prove vanishing theorems of secondary characteristic classes of foliations. However, while they localized the tangential part called the Godbillon or Weil measures, we will localize the transverse part, primary characteristic classes of the normal bundle of foliations, based on equivariant basic cohomology defined in a paper by the first and third author~\cite{GT2010}, see also \cite{Toeben}.

\subsection{Main result: an ABBV type localization formula for Killing foliations}

Our foliated ABBV type formula applies to Killing foliations, of which the main example in this paper is the orbit foliation $\F$ of a nonsingular Killing vector field $\xi$ on a $(2n+1)$-dimensional oriented compact Riemannian manifold $(M,g)$ with only finitely many closed $\xi$-orbits. For simplicity, let us state the formula for this case -- see Theorem \ref{thm:loc} for the general statement. Let $T$ be the closure of the flow generated by $\xi$ in $\Isom (M,g)$, which is a torus. Let $\mft = \Lie(T)$. Take $b\in \mft$ so that the fundamental vector field of $b$ equals $\xi$ and let $\mfa = \mft/\RR b$. Then $\mfa$ acts transversely on $(M,\F)$ (Definition~\ref{def:transverseaction}). Let $H(M,\F)$ and $H_{\mfa}(M,\F)$ denote the basic respectively $\mfa$-equivariant basic cohomology of $\F$  (see Section \ref{sec:eqbascoh}). For a $1$-form $\eta$ on $M$ such that $\iota_{\xi}d\eta=0$, we have the transverse integration operator $\int_{(\F,\eta)}: H(M,\F) \to \RR; [\sigma] \mapsto \int_{M} \eta \wedge \sigma$ associated to $\eta$ (Definition~\ref{defn:trint}), which extends to equivariant basic cohomology $\int_{(\F,\eta)} : H_{\mfa}(M,\F) \to S(\mfa^{*})$.

\begin{thmintro}[Foliated ABBV-type localization formula for nonsingular Killing vector fields with isolated closed orbits]\label{thmintro:1}
For any $\sigma\in H_{\mfa}(M,\F)$, we have an equality
\[
\int_{(\F,\eta)}\sigma = (-2\pi)^{n} \sum_{k} l_k\cdot  \frac{i^*_{L_k}\sigma}{\prod_j \alpha_{j}^{k}},
\]
in the fraction field of $S(\mfa^{*})$, where $i_{L_{k}}:L_k\to M$ are the closed $\xi$-orbits, $\{\alpha_{j}^{k}\}_{j=1}^{n} \subset \mfa^{*}$ are the weights of the isotropy $\mfa$-action at $L_{k}$ and $l_k=\int_{L_k}\eta$.
\end{thmintro}

The general ABBV-type theorem, Theorem \ref{thm:loc}, is proved in Section \ref{sec:proofABBV}. This result can be seen as an improvement of a localization result obtained by the third author \cite[Theorem 7.1]{Toeben}. See Remark~\ref{rem:comparison} for a comparison.

\subsection{Application I: localization of the volume of Sasakian manifolds}

We can apply the ABBV type formula in the last section to localize the volume of Sasakian manifolds, or more generally K-contact manifolds, to the union of closed Reeb orbits. Recall that a K-contact structure on a manifold $M$ is a contact metric structure $(\xi,\eta,J,g)$ whose Reeb flow preserves the Riemannian metric $g$. For simplicity of the presentation, we assume that the action of the closure $T$ of the Reeb flow admits only finitely many $S^{1}$-orbits. For $v \in \mft$, denote its fundamental vector field on $M$ by $v^{\#}$.
With the same notation as in the last theorem, with $\xi$ the Reeb vector field and $\eta$ the contact form, we obtain the following application of Theorem 1 (See Section 5.6):

\begin{thmintro}\label{thmintro:2}
Let $(M,\xi,\eta,J,g)$ be a $(2n+1)$-dimensional compact K-contact manifold with only finitely many closed Reeb orbits $L_1$, $\ldots$, $L_N$. Denote the weights of the transverse isotropy $\mfa$-representation at $L_{k}$ by $\{\alpha^{k}_j\}_{j=1}^{n} \subset \mfa^*$ for $k=1$, $\ldots$, $N$. Then, the volume of $M$ is given by
\begin{equation*}
\Vol(M,g) = (-1)^n\frac{\pi^n}{n!} \sum_{k=1}^N l_{k}\cdot \frac{\eta|_{L_{k}}(v^{\#})^n}{\prod_j \alpha^{k}_j(v+\RR b)},
\end{equation*}
where $l_{k}=\int_{L_{k}} \eta$ is the length of the closed Reeb orbit $L_{k}$, and the fractions on the right hand side are considered as rational functions in the variable $v\in \mft$. The total expression is independent of $v \in \mft$.
\end{thmintro}

Two motivations to find computable formulas for the volume of Sasakian manifolds were given by Martelli-Sparks-Yau \cite{MSY2,MSY}: they calculated that, restricted to a space of Sasakian metrics on a manifold $M$, the Einstein-Hilbert functional equals, up to a constant, the volume functional (see \cite[Eq.\ (3.14)]{MSY}). Thus understanding the volume of Sasaki manifolds is closely related to the problem of finding Sasaki-Einstein metrics on $M$. A second motivation comes from string theory: restricting to $5$-dimensional Sasaki-Einstein manifolds $M$, the AdS/CFT correspondence is a conjectural duality between type IIb string theory on ${\mathrm{AdS}_5}\times M$ and an $N=1$ superconformal field theory on the conformal boundary of ${\mathrm{AdS}}_5$. Under this duality, the volume of $M$ corresponds to the central charge $a$ of the field theory. With these motivations in mind, Martelli-Sparks-Yau applied a noncompact orbifold version of the classical ABBV localization theorem to an orbifold resolution of the K\"ahler cone of a Sasakian manifold $M$ in order to find a formula for the volume. Our method is purely intrinsic and does not make use of the K\"ahler cone, and as such does not need any technical assumptions on the existence of certain metrics on the cone (see \cite[Footnote 35]{MSY}).

As an important special case, we apply Theorem \ref{thmintro:2} to obtain a computable formula for the volume of toric Sasakian manifolds (see Theorem \ref{thm:localizationvolumeSasakiantoric}). Together with the observation of Martelli-Sparks-Yau \cite{MSY} that the volume of a toric Sasakian manifold equals, up to a constant, the volume of its momentum polytope (see Proposition~\ref{prop:MSY}), our formula also follows from Lawrence's formula (Theorem~\ref{thm:Lawrence}) for the volume of a simple polytope.

A further class of Sasakian manifolds to which we apply our formula is that of (deformations of) homogeneous Sasakian manifolds. We obtain in Corollary \ref{cor:volumehomogeneouscase} a formula for the Sasakian volume in this case which, to our knowledge, has not appeared in the literature before.

\subsection{Application II: localization of secondary characteristic classes of Riemannian foliations}\label{sec:sch}

The classifying space $FR\Gamma_{q}$ of codimension $q$ Riemannian foliations with framed normal bundles is constructed based on Haefliger's work \cite{Haefliger2,Haefliger3}. Secondary characteristic classes of codimension $q$ Riemannian foliations with framed normal bundle (Lazarov-Pasternack \cite{LazarovPasternack1} and Morita~\cite{Morita}) are the pull-back of certain cohomology classes of the classifying space $FR\Gamma_{q}$ by the classifying map (see~\cite{Hurder} for a survey). The topology of $FR\Gamma_{q}$ is rather complicated. Instead of analysing it directly, it is classical to investigate the behavior of secondary classes on well-chosen examples of foliations to retrieve cohomological information of $FR\Gamma_{q}$. This is the motivation for the computation of secondary classes of examples of foliations.

Here we will apply our localization formula to compute these invariants. Given a codimension $q$ Riemannian foliation $\F$ on a smooth manifold with framed normal bundle, we have a characteristic homomorphism
\[
\Delta_{\F} : H(RW_{q}) \longrightarrow H(M;\RR),
\]
where $RW_{q}$ is the dga of the universal characteristic classes consisting of Pontryagin classes $p_{i}$, Euler class $e$ and their transgressions $h_{i}$ (see Section~\ref{sec:charclass}). The elements in the image of $\Delta_{\F}$ are called the secondary characteristic classes of $\F$. It is simple but remarkable that the transgression $h_{i}$ of Pontryagin classes are relatively closed (Lemma~\ref{lem:hI}). This fact allows us to apply our localization formula (Corollary~\ref{cor:locchar}) to compute some characteristic numbers of Killing foliations. The secondary classes of transversely K\"{a}hler foliations can be localized similarly.

For simplicity, we will state the localization formula for transversely K\"{a}hler flows. For a complex codimension $m$ transversely K\"{a}hler foliation $\F$ on $M$ with framed normal bundle, Matsuoka-Morita \cite{Matsuoka} defined a characteristic homomorphism
\[
\Delta_{\F} : H(KW_{m}) \longrightarrow H(M;\RR),
\]
where $KW_{m}$ is a differential graded algebra given by
\[
KW_{m} = \bigwedge (u_1, \ldots, u_{m}) \otimes \left( \RR[s_1, \ldots, s_{m}, \Phi]/\{P \mid \deg P> m\} \right),
\]
which consists of the universal Chern classes, its transgression and the transverse K\"{a}hler class. Below we use the multi-index notation $u_{I}s_{J} := u_{i_{1}} \wedge \cdots \wedge u_{i_{k}} \wedge s_{j_{1}} \wedge \cdots \wedge s_{j_{\ell}}$ for $I = \{i_{1}, \ldots, i_{k}\}$ and $J = \{ j_{1}, \ldots, j_{\ell}\}$. We will denote $s_{J}(\nu \F, \F)=\Delta_{\F}(s_{J})$, $u_{I}(\F)=\Delta_{\F}(u_{I})$ and $u_{I}s_{J}(\F)=\Delta_{\F}(u_{I}s_{J})$.

Let $(M,\F)$ be a manifold with a transversely K\"{a}hler foliation of codimension $m$ such that $\wedge^{m,0}\nu^{*} \F = \wedge^{m} (\nu^{1,0} \F)^{*}$ is topologically trivial. We fix a trivialization $\varphi$ of $\wedge^{m,0}\nu^{*} \F$. Using $\varphi$ we obtain a $1$-form $u_{1}(\F)$ on $M$ which is a primitive of the first Chern form of $\nu^{1,0} \F$ (see Definition~\ref{def:u1}). Thus we can extend the definition of secondary characteristic classes $u_{1}s_{J}(\F)$ to transversely K\"{a}hler foliations such that $\wedge^{m,0}\nu^{*} \F$ is topologically trivial.

\begin{thmintro}
Let $(M,\F)$ be a compact manifold with an orientable taut transversely K\"{a}hler foliation of dimension one and complex codimension $m$. Assume that $\wedge^{m,0}\nu^{*} \F$ is trivial as a topological line bundle and that $\F$ has only finitely many closed leaves $L_{1}$, $\ldots$, $L_{N}$. For a given multiindex $J=\{j_1,\ldots,j_l\}$ with $j_1+\cdots+j_l= 2m$, we have
\begin{equation*}
\int_{M} u_{1} s_J (\F) = \sum_{k=1}^{N} \left(\int_{L_{k}} u_{1}(\F) \right) \frac{i^*_{L_{k}}s_{J,\mfa}(\nu \F, \F)}{i^{*}_{L_{k}}s_{m,\mfa}(\nu \F, \F)},
\end{equation*}
where the $L_{k}$ are the isolated closed leaves of $\F$ and $s_{J,\mfa}(\nu \F, \F)$ is the equivariant characteristic form of $\F$ associated to $s_{J}$. In particular, in the case where $J=\{m\}$, we obtain
\begin{equation}
\int_{M} u_{1}s_{m}(\F) = \sum_{k} \int_{L_{k}} u_{1}(\F).
\end{equation}
\end{thmintro}

As a consequence (see Corollary~\ref{cor:AsukeExample}), we recover a formula for some secondary characteristic numbers of certain foliations on $S^{2n+1}$ due to Bott~\cite{Bott0}, Baum-Bott~\cite{BaumBott} and Asuke \cite{Asuke3}.

\newtheorem*{org}{Organization of the article}
\begin{org}
Section~\ref{sec:HaMF} is devoted to recall fundamentals of equivariant basic cohomology. The main result, an ABBV type localization formula, is stated in Section~\ref{sec:ABBV}, together with necessary facts on transverse integration operators and Killing foliations. It is proved in Section \ref{sec:proofABBV}, modulo the introduction of the equivariant basic Thom homomorphism which is contained in the Appendix. In Section~\ref{sec:K-contact}, we apply the localization formula to the Reeb flow of K-contact manifolds. We present examples of computations on toric Sasakian manifolds and deformations of homogeneous Sasakian manifolds in Section~\ref{sec:ex}. Section~\ref{sec:second} is devoted to apply the localization formula to secondary characteristic classes of Riemannian and transversely K\"{a}hler foliations.
\end{org}

\newtheorem*{ack}{Acknowledgements}
\begin{ack}
We are grateful to Lana Casselmann for pointing out a mistake in the proof of Theorem~\ref{thm:loc} in a previous version. This paper was partially written during the stay of the second author at Centre de Recerca Matem\`{a}tica (Bellaterra, Spain), Institut Mittag-Leffler (Djursholm, Sweden) and Institut des Hautes \'{E}tudes Scientifiques (Bures-sur-Yvette, France); he is very grateful for their hospitality.
\end{ack}

\section{Equivariant basic cohomology}\label{sec:HaMF}

\subsection{Transverse actions on foliated manifolds}\label{sec:transverseaction}

Let us recall the notion of a transverse action on a foliated manifold introduced in~\cite{Alvarez Lopez}, which is essential to define equivariant basic cohomology in the next section.

Let $M$ be a smooth manifold, and $\Xi(M)$ the Lie algebra of vector fields on $M$. Given a foliation $\F$ on $M$, the Lie algebra of vector fields on $M$ that are tangent to the leaves of $\F$ is denoted by $\Xi(\F)$. A vector field $X$ on $M$ is said to be {\em foliated} if $[X,Y]\in\Xi(\F)$ for all $Y\in \Xi(\F)$. A vector field is foliated if and only if its flow maps leaves of $\F$ to leaves of $\F$ (see~\cite[Proposition~2.2]{Molino}). We call the projection of a foliated field $X$ to $C^{\infty}(TM/T\F)$ a {\em transverse} field. It is easy to see that the set $L(M,\F)$ of foliated fields is an ideal in $\Xi(M)$, and hence the set
\[
l(M,\F)=L(M,\F)/\Xi(\F)
\]
of transverse fields is a Lie algebra with Lie bracket induced from $L(M,\F)$.

\begin{defn}[{\cite[Section 2]{Alvarez Lopez}}]\label{def:transverseaction}
A  {\em transverse action} of a finite-dimen\-sional real Lie algebra $\mfg$ on the foliated manifold $(M,\F)$ is a Lie algebra homomorphism $\mfg\to l(M,\F)$.
\end{defn}
If $\F$ is the trivial foliation by points, this notion coincides with the usual notion of an infinitesimal action on the manifold $M$.

\subsection{Definition of equivariant basic cohomology}\label{sec:eqbascoh}

Let us recall the definition of equivariant basic cohomology for foliated manifolds $(M,\F)$ with a transverse action (\cite[Definition 3.13]{GT2010}), which is a generalization of ordinary equivariant cohomology of Lie group actions on manifolds.

Given a foliation $\F$ on a manifold $M$, recall that the {\em basic de Rham complex} of $(M,\F)$ is defined by
\[
\Omega(M,\F)=\{\sigma\in \Omega(M)\mid \iota_X\sigma=L_X\sigma=0\text{ for all }X\in \Xi(\F)\}
\]
equipped with the restriction of the ordinary exterior differential $d$. An element of $\Omega(M,\F)$ is called a {\em basic} differential form and the cohomology
\[
H(M,\F) := H(\Omega(M,\F),d)
\]
is called the {\em basic cohomology} of $(M,\F)$. It was introduced by Reinhart \cite{Reinhart}. It may be regarded as a replacement for the de Rham cohomology of the leaf space which is well-defined also in case the leaf space is not a differentiable manifold.

\begin{rem}
Take a foliated chart $(x,y)=(x_{1}, \ldots, x_{p}, y_{1}, \ldots, y_{q})$ so that the leaves are defined by $y = const$. Then it is easy to see that any basic form $\alpha$ is of the form $\alpha = \sum f_{I}(y) dy_{I}$, i.e., the pull-back of a differential form on the local space of leaves. In this sense basic cohomology is a natural generalization of the de Rham cohomology of the leaf space.
\end{rem}

It is not difficult to check (see {\cite[Proposition~3.12]{GT2010}}) that we obtain well-defined derivations $\iota_X$ and $L_X$ on $\Omega(M,\F)$ for all $X\in l(M,\F)$, which induces a structure of an $l(M,\F)$-differential graded algebra on $\Omega(M,\F)$. Namely, it satisfies the usual compatibility relations $d^2=0,\, \iota_X^2=0,\, L_X=d\iota_X+\iota_Xd,\, [d,L_X]=0,\, [L_X,\iota_Y]=\iota_{[X,Y]}$ and $[L_X,L_Y]=L_{[X,Y]}$. Thus a $\mfg$-action $\mfg \to l(M,\F)$ on $(M,\F)$ induces a structure of a $\mfg$-differential graded algebra on $\Omega(M,\F)$. Then the {\em Cartan complex} of $\Omega(M,\F)$ \cite{Cartan2} (see also \cite{GS}) is defined by
\[
\Omega_{\mfg}(M,\F):=(S(\mfg^*)\otimes \Omega(M,\F))^\mfg,
\]
where the superscript denotes the subspace of $\mfg$-invariant elements, namely, those $\sigma\in S(\mfg^*)\otimes \Omega(M,\F)$ for which $L_X\sigma=0$ for all $X\in \mfg$. The grading of $\Omega(M,\F)$ is defined by $\Omega_\mfa^k(M,\F)=\bigoplus_{k=2u+v}(S^{u}(\mfg^{*})\otimes \Omega^v(M,\F))^\mfg$. The differential $d_\mfg$ of the Cartan complex $\Omega_{\mfg}(M,\F)$ is given by
\[
(d_\mfg \sigma)(X)=d(\sigma(X))-\iota_X(\sigma(X)),
\]
where $\sigma \in \Omega_{\mfg}(M,\F)$ is regarded as a $\mfg$-equivariant polynomial map $\mfg\to \Omega(M,\F)$; $X \mapsto \sigma(X)$.

\begin{defn}[{\cite[Section 3.6]{GT2010}}]\label{def:eqformaltraction}
The {\em equivariant basic cohomology} of a transverse $\mfg$-action on $(M,\F)$ is defined as
$$
H_\mfg(M,\F):=H(\Omega_\mfg(M,\F), d_{\mfg})\;.
$$
\end{defn}
If $\F$ is the trivial foliation by points, the transverse action is just an infinitesimal action on the manifold $M$ and equivariant basic cohomology reduces to ordinary equivariant de Rham cohomology.

\section{An ABBV-type localization formula in equivariant basic cohomology}\label{sec:ABBV}

\subsection{Definition of Riemannian foliations}
To have a foliated version of an ABBV type localization formula, one needs good transverse dynamics of foliations. Here we recall the definition of Riemannian foliations, which are the main object in this paper.

A {\em Riemannian Haefliger cocycle} of codimension $q$ on a manifold $M$ is a quadruple $(\{U_i\},\{g_{i}\},\{\pi_i\},\{\gamma_{ij}\})$ consisting of
\begin{enumerate}
\item an open covering $\{U_i\}$ of $M$,
\item submersions $\pi_i:U_i\to \RR^q$,
\item Riemannian metrics $g_{i}$ on $\pi_{i}(U_{i})$,
\item transition maps $\gamma_{ij}:\pi_j(U_i\cap U_j)\to \pi_i(U_i\cap U_j)$ such that $\pi_i=\gamma_{ij}\circ \pi_j$ and $\gamma_{ij}^{*}g_{i}=g_{j}$.
\end{enumerate}
Two Riemannian Haefliger cocycles on $M$ are said to be equivalent if their union is a Riemannian Haefliger cocycle on $M$. A {\em Riemannian foliation} of codimension $q$ is defined to be an equivalence class of Riemannian Haefliger cocycles of codimension $q$.

Given a codimension $q$ Riemannian foliation on $M$, the connected components of the union of the fibers of $\pi_{i}$ define a codimension $q$ foliation on $M$, which is denoted by $\F$. We will see that the normal bundle $\nu \F$ admits a natural metric. We obtain a metric on $\nu \F|_{U_{i}}$ by pulling back the metric $g_{i}$ on $\pi_{i}(U_{i})$ by the isomorphism $(\pi_{i})_{*} : \nu_{x}\F \to T_{\pi_{i}(x)}\RR^{q}$ at each point $x\in U_{i}$. Since $\{g_{i}\}$ is invariant under the transition maps $\gamma_{ij}$, these metrics give rise to a well-defined metric $g$ on $\nu \F$, which satisfies $L_{X}g=0$ for any $X \in C^{\infty}(T\F)$. Note that one can recover the Riemannian foliation from $\F$ and $g$. In this article, the pair of $\F$ and $g$ is also called a Riemannian foliation on $M$.

In general, for a codimension $q$ foliated manifold $(M,\F)$, a metric $g$ on $\nu \F$ is called a {\em transverse} metric on $(M,\F)$ if it satisfies $L_{X}g=0$ for any $X \in C^{\infty}(T\F)$. A Riemannian metric on $M$ is {\em bundle-like} with respect to $\F$ if the metric induced on $\nu \F$ via the identification $\nu \F \cong (T\F)^{\perp}$ is a transverse metric on $(M,\F)$. It is easy to see that any Riemannian foliation admits a compatible bundle-like metric (see~\cite[Proposition 3.3]{Molino}). The following is the notion of the completeness in the transverse direction of Riemannian foliations.

\begin{defn}
A Riemannian foliation $\F$ on a connected manifold $M$ is {\em transversely complete} if there exists a bundle-like metric on $(M,\F)$ which is transversely complete, namely, any maximal geodesic orthogonal to the leaves is defined on all of $\RR$.
\end{defn}

\begin{rem}
Transverse completeness in this definition is different from the one in \cite[Definition 4.1]{Molino}. \'{A}lvarez-Masa~\cite[Proposition 15.1]{AlvarezMasa} proved that a transversely complete bundle-like metric becomes complete after a conformal change on the leaves.
\end{rem}

\subsection{Integration of basic differential forms}
To formulate the localization formula for equivariant basic cohomology, one needs to integrate basic cohomology classes of maximal degree. But, in general, there is no good notion of fundamental classes for the leaf spaces. Instead we will use a classical method in foliation theory based on the pairing of differential forms on foliated manifolds. Namely, to integrate a basic cohomology class, we multiply it with a leafwise volume form and then integrate it over the ambient manifold. This can be interpreted in terms of the spectral sequence associated to a foliation.

Let $\F$ be a foliation of dimension $p$ and codimension $q$ on an oriented manifold $M$ of dimension $n=p+q$.
\begin{defn}
A $p$-form $\eta$ on $M$ is {\em relatively closed} if
$$
d\eta(v_1,\ldots,v_{p+1})=0
$$
whenever $p$ of the $p+1$ vectors $v_i$ are tangent to the foliation.
\end{defn}
\begin{ex}
The contact form of a K-contact manifold is relatively closed with respect to the orbit foliation of the Reeb flow, see Section \ref{subsec:volume}. The same holds for certain Chern-Simons forms of Riemannian foliations, see Section \ref{subsec:loccharnumbers}.
\end{ex}

Let $\eta$ be a compactly supported relatively closed $p$-form on $(M,\F)$. Consider the map
\begin{equation}\label{eq:transverseintegration}
\int_{(\F,\eta)}:\Omega^q(M,\F)\longrightarrow\RR;\quad \sigma\longmapsto \int_M \eta\wedge\sigma.
\end{equation}
The following is well-known.
\begin{prop}
$\int_{(\F,\eta)}$ descends to a map $\int_{(\F,\eta)}:H^q(M,\F)\to\RR$.
\end{prop}
\begin{proof}
We show that $\int_{(\F,\eta)}$ is trivial on the space $d\Omega^{q-1}(M,\F)$ of exact forms: For a basic $(q-1)$-form $\alpha$, we have $\int_{(\F,\eta)} d\alpha=\int_M d(\eta\wedge\alpha)-\int_M d\eta\wedge\alpha=0$: the first summand is zero due to Stokes theorem, and $d\eta\wedge\alpha=0$ because $\eta$ is relatively closed and $\alpha$ basic.
\end{proof}

\begin{defn}\label{defn:trint}
The map $\int_{(\F,\eta)}:H^q(M,\F)\to\RR$ induced from~\eqref{eq:transverseintegration} on cohomology is called the {\em transverse integration operator with respect to $\eta$}.
\end{defn}

Here we briefly discuss the ambiguity of transverse integration operators on foliated manifolds which comes from the choice of $\eta$. Let $\Omega^{k}_{c} (\F) = C^{\infty}_{c}(\wedge^{k} T^{*}\F)$, where the subscript $c$ indicates compact support. Then $\Omega_{c} (\F)$ is a differential complex with the usual differential $d_{\F}$ on each leaf.
\begin{defn}
The cohomology $H(\Omega_{c} (\F), d_{\F})$ is called the {\em compactly supported leafwise cohomology} of $(M,\F)$.
\end{defn}
\begin{rem}
The leafwise cohomology $H_{c}(\F)$ is often of infinite dimension and non-Hausdorff with the $C^{\infty}$-topology (see~\cite{Alvarez Lopez Hector}).
\end{rem}
Let $r_{\F} : \Omega^{k}_{c}(M) \to \Omega^{k}_{c}(\F)$ be the restriction map. Given a compactly supported relatively closed $p$-form $\eta$, it is easy to see that the transverse integration operator $\Omega^q(M,\F)\to\RR$ with respect to $\eta$ on the differential form level depends only on the leafwise cohomology class of the tangential part $r_{\F}(\eta)$. Let $E^{s,t}_r$ (resp.\ $E^{s,t}_{r,c}$) be the spectral sequence associated to $\F$, which is defined by a natural filtration on $\Omega(M)$ (resp.\ $\Omega_{c}(M)$) and converges to $H(M)$ (resp.\ $H_{c}(M)$). It is known that
\begin{align*}
E^{0,t}_{1} &=H^{t}(\F), &&&
E^{s,0}_{1} &=\Omega^{s}(M,\F), &&&
E^{s,0}_{2} &=H^{s}(M,\F).
\end{align*}
We have analogous equalities in the compactly supported case. Then $E^{0,p}_{2,c}$ is the space of leafwise cohomology classes represented by a compactly supported relatively closed $p$-form. The wedge product on $\Omega(M)$ induces a pairing of $E_{2,c}^{0,p}$ and $E_2^{q,0}$ that reads
\begin{equation}\label{eq:pairing}
E^{0,p}_{2,c} \otimes H^q(M,\F) \longrightarrow H^{p+q}_{c}(M) \cong \RR
\end{equation}
(see e.g. \cite[Section 2]{Sarkaria}). In summary, we have the following.
\begin{prop}
$E_{2,c}^{0,p}$ is the space of compactly supported leafwise cohomology classes represented by a relatively closed $p$-form.
Via the above pairing~\eqref{eq:pairing}, they correspond one-to-one to transverse integration operators.
\end{prop}

Originally Kamber and Tondeur~\cite{KamberTondeur} developed Poincar\'{e} duality theory for basic cohomology and the spectral sequence of tense Riemannian foliations on closed manifolds. By theorems of Masa~\cite{Masa} and Dom\'{i}nguez~\cite{Dominguez} obtained later, their results can be applied to general Riemannian foliations on closed manifolds. The second author~\cite{Nozawa} generalized their results for complete Riemannian foliations whose space of leaf closures is compact. The case of Riemannian foliations of dimension one was established by the second author and Royo Prieto~\cite{NozawaRoyoPrieto}. A part of these works is summarized as follows.
\begin{thm}[\cite{Masa,Dominguez,KamberTondeur,Nozawa,NozawaRoyoPrieto}]\label{thm:amb}
Let $(M,\F)$ be an orientable connected manifold with an orientable transversely complete Riemannian foliation whose space of leaf closures is compact. Then one of the following two cases occurs:
\begin{enumerate}
\item $H^{q}(M,\F) \cong \RR$, $E^{0,p}_{2,c} \cong \RR$ and the pairing~\eqref{eq:pairing} is nondegenerate. In particular, $(M,\F)$ admits a nontrivial transverse integration operator, uniquely up to constants.
\item $H^{q}(M,\F) = \{0\}$, $E^{0,p}_{2,c} = \{0\}$ and the pairing~\eqref{eq:pairing} is trivial.  In particular, $(M,\F)$ admits no nontrivial transverse integration operator.
\end{enumerate}
\end{thm}
\begin{rem}
It is well known that, if $M$ is compact, by a theorem of Masa~\cite{Masa}, the first case occurs if and only if $(M,\F)$ is {\em taut}; namely, $M$ admits a Riemannian metric $g$ such that every leaf of $\F$ is a minimal submanifold of $(M,g)$. Note that, for transversely complete Riemannian foliations whose space of leaf closures is compact, the only if part is true (see~\cite[Corollary~1.11]{Nozawa}) while the if part is not true anymore (see~\cite[Example~9.2]{Nozawa}).
\end{rem}

\subsection{Killing foliations}\label{sec:traction}\label{sec:tractiononKilling}

Here we recall the notion of Killing foliations. A Killing foliation admits a natural transverse action with respect to which we localize.

The Molino sheaf $\mathcal{C}$ of a Riemannian foliation $(M,\F)$ is a locally constant sheaf of Lie algebras, whose stalks consist of certain local transverse vector fields. Precisely, a stalk of $\mathcal{C}$ consists of germs of local transverse vector fields on $(M,\F)$ whose natural lifts to the orthonormal frame bundle $M^{1}$ of $(M,\F)$ commute with any global transverse field of $(M^{1},\F^{1})$, where $\F^{1}$ is the canonical horizontal lift of $\F$ (see \cite[Section~4]{Molino}).
\begin{defn}
A Riemannian foliation is called a \emph{Killing foliation} if its Molino sheaf is globally constant.
\end{defn}
\begin{ex}
Riemannian foliations on simply connected manifolds are Killing. By a theorem of Molino-Sergiescu~\cite[Th\'{e}or\`{e}me~A]{MolSer}, any taut orientable $1$-dimensional Riemannian foliation on a compact manifold is Killing. In particular, the orbit foliation of a nonsingular Killing vector field on a compact Riemannian manifold is Killing.
\end{ex}

Any global section of the Molino sheaf $\mathcal{C}$ of a Riemannian foliation $(M,\F)$ is a transverse field on $(M,\F)$ which commutes with any global transverse field of $(M,\F)$. So the space $\mfa$ of global sections of $\mathcal C$ is central in $l(M,\F)$, hence it is an abelian Lie algebra acting transversely on $(M,\F)$. For a Killing foliation $\F$, we will, following~\cite{GT2010}, call the Lie algebra $\mfa$ the {\em structural Killing algebra} of $\F$. By \cite[Theorem~5.2]{Molino}, the orbits of the leaves under the action of the structural Killing algebra are the leaf closures, cf.\ also~\cite[Section~4.1]{GT2010}. In summary, we have the following.
\begin{prop}
For a Killing foliation $(M,\F)$, the space $\mfa$ of global sections of the Molino sheaf, the structural Killing algebra, is abelian and acts transversely on $(M,\F)$ in a canonical way. Its orbits are the leaf closures of $\F$.
\end{prop}

Sergiescu's orientation sheaf is trivial for a Killing foliation by definition. Then, if $M$ is connected and the space of leaf closures of $(M,\F)$ is compact, by~\cite[Th\'eor\`eme I]{Sergiescu}, we have $H^{q}(M,\F) \cong \RR$, where $q=\cod \F$. Combining this result with Theorem~\ref{thm:amb}, we get the following.
\begin{cor}
An orientable transversely complete Killing foliation on a connected orientable manifold whose space of leaf closures is compact admits a nontrivial transverse integration operator which is unique up to a constant.
\end{cor}

\subsection{The localization formula}\label{sec:loc}

The analogue of the classical localization formula for transverse integration operators with respect to a relatively closed form reads as stated below. Let $C$ be the union of the closed leaves of $\F$, and denote the connected components of $C$ by $C_k$. The $C_k$ are submanifolds of $M$.

\begin{thm}[ABBV-type localization formula for transverse integration operators]\label{thm:loc}
Let $\F$ be a transversely complete orientable Killing foliation on an oriented Riemannian manifold $M$ whose space of leaf closures is compact. Furthermore, let $\eta$ be a compactly supported relatively closed $p$-form. Then for any $\sigma\in H_\mfa(M,\F)$, where $\mfa$ is the structural Killing algebra of $\F$, we have an equality
$$
\int_{(\F,\eta)}\sigma=\sum_{k} \int_{(\F|C_{k},i_{C_k}^*\eta)} \frac{i^*_{C_{k}}\sigma}{e_\mfa(\nu C_{k},\F)} =\sum_{k} l_{k}\int_{C_{k}/\F} \frac{i^*_{C_{k}}\sigma}{e_\mfa(\nu C_{k},\F)},
$$
in the fraction field of $S(\mfa^{*})$, where $L_{k}$ is a leaf of $\F|C_{k}$ without holonomy, $l_{k}=\int_{L_{k}}\eta$ and $e_\mfa(\nu C_{k},\F)$ is the equivariant basic Euler form of the normal bundle of $C_{k}$ (see Section \ref{sec:eqcharclasses}).
\end{thm}

\begin{cor}\label{cor:isolatedclosedleaves}
If, in the situation of Theorem \ref{thm:loc}, the closed leaves of $\F$ are isolated, then we have for any $\sigma\in H_{\mfa}(M,\F)$ that
\[
\int_{(\F,\eta)}\sigma = (-2\pi)^{q/2} \sum_k l_k\cdot  \frac{i^*_{L_k}\sigma}{\prod_j \alpha_{j}^{k}},
\]
where $q$ is the codimension of $\F$, $i_{L_{k}}:L_k\to M$ are the closed leaves of $\F$,  $\{\alpha_{j}^{k}\}_{j=1}^{q/2} \subset \mfa^{*}$ are the weights of the transverse isotropy $\mfa$-action at $L_{k}$ and $l_k=\int_{L_k}\eta$.
\end{cor}

\begin{ex}
A Riemannian foliation on a compact manifold is always transversely complete and the space of leaf closures is compact.
\end{ex}

\begin{ex}
The orbit foliation $\F$ of a nonsingular Killing vector field $X$ on a complete Riemannian manifold $(M,g)$ satisfies the assumption if and only if the space of the orbit closures is compact. Here the closure $T$ of the one parameter subgroup generated by $T$ is an abelian subgroup of the isometric group of $(M,g)$. The space of leaf closures of $\F$ is nothing but the orbit space of the $T$-action on $M$.
\end{ex}

\begin{rem}
Since $C_{k}$ is a union of closed leaves, $C_{k}/\F$ is an orbifold. The integration $\int_{C_{k}/\F}$ is defined as follows: We choose orientations of $C_{k}$ and $\nu C_{k}$ (see \cite[Corollary 4.8]{Toeben} for orientability) compatible with the orientation of $M$. The integration on the right side is taken with respect to the orientation of $C_{k}$. The integration of a volume form $\sigma$ on an orbifold, as it appears on the right hand side, is defined in terms on the orbifold fundamental class similarly to the manifold case. Via a partition of unity, one can assume the support of $\sigma$ to lie in the domain of an orbifold chart $U=\widetilde U/\Gamma_{k}$, where $\widetilde U$ is open in $\RR^n$ and $\Gamma$ is a finite group acting on $\widetilde U$. By definition $\sigma$ has a lift $\widetilde\sigma$ in $\widetilde U$. We define $\int_U\sigma:=\frac{1}{|\Gamma|}\int_{\widetilde U}\widetilde \sigma$ where the right hand is the usual integral.
\end{rem}

\begin{rem}\label{rem:comparison}
This theorem can be regarded as an improvement of \cite[Theorem 7.2]{Toeben}. The ambiguity in the choice of transverse integration operator mentioned in Theorem~\ref{thm:amb} results there in constants on the right side of \cite[Theorem 7.2]{Toeben} which are hard to determine. In contrast, because of the coherent use of a relatively closed form $\eta$ in both sides of the equation, here the constants turn out to be $1$ (see Section \ref{app:basicthom}).

This localization formula is also more versatile. Whereas the corresponding formula in \cite{Toeben} was used to compute basic primary characteristic numbers, we can now also compute secondary characteristic numbers and the volume of K-contact manifolds as we will see later.
\end{rem}
Our proof of this theorem runs along the lines of the classical proof of Atiyah and Bott \cite{AtiyahBott}, but makes use of several transverse analogues of classical objects, like the (equivariant) basic Thom isomorphism (see the appendix).

\begin{rem} The hypothesis of Theorem~\ref{thm:loc} that the foliation is Killing is essential in the proof. For example, for an arbitrary isometric transverse action of an abelian Lie algebra on a Riemannian foliation, we do not know whether a Borel type localization theorem \cite[Theorem 5.2]{GT2010} holds (see, for example, \cite[Proof of Proposition 3.6]{GT2010}).
\end{rem}

\section{Proof of the ABBV-type localization formula}\label{sec:proofABBV}

\subsection{Some facts on transverse integration operators}\label{app:basicthom}

Let $M$ be an oriented connected manifold. Let $\F$ be an oriented transversely complete Riemannian foliation. Let $N$ be a connected component of the union of closed leaves of $\F$. By transverse completeness, the normal exponential map $\exp : \nu \F \to M$ is well-defined. By using $\exp$, we can take a saturated tubular neighborhood $U$ of $N$ with a projection ${u} : U \to N$ which maps the leaves of $(U,\F)$ to the leaves of $(N,\F)$. Let ${u}_*:\Omega^{r+\bullet}_{cv}(U)\to \Omega^{\bullet}(N)$ be the integration along the fibers (see for instance in \cite[I.7.12]{GHV}), where $\Omega_{cv}$ denotes the vertically compactly supported de Rham complex and $r$ is the dimension of the fibers of ${u}$. Since ${u} :(U,\F)\to (N,\F)$ is foliated, the fiber integration ${u}_*$ descends to a chain map ${u}_*:\Omega^{r+\bullet}_{cv}(U,\F)\to \Omega^{\bullet}(N,\F)$ of basic complexes (see \cite[Proposition 4.1]{Toeben}).

We will use the following well known fact.
\begin{lemma}
Let $\mathcal{I} : \Omega_{cv}^{k}(U \times [0,1]) \to \Omega_{cv}^{k-1}(U)$ be the integration along the $[0,1]$-fibers. Then, we have
\[
j_{1}^{*} \alpha - j_{0}^{*} \alpha = d \mathcal{I} \alpha + \mathcal{I} d \alpha
\]
for any $\alpha \in \Omega_{cv}^{k}(U \times [0,1])$, where $j_{t} : U \to U \times [0,1]$ is defined by $j_{t}(x)=(x,t)$ for $t=0$ and $1$.
\end{lemma}

\begin{lemma}\label{lem:projformula}
Let $p=\dim \F$ and $q=\codim \F$. For any compactly supported, relatively closed $p$-form $\eta$ on $(M,\F)$ and any $\sigma\in \Omega_{cv}^{q}(U,\F)$, we have
\begin{equation}\label{eq:projformula}
{u}_{*}(\eta\wedge \sigma)=i^*\eta\wedge {u}_{*}\sigma + d\theta
\end{equation}
for some $\theta \in \Omega_c(N)$.
\end{lemma}

\begin{proof}
By the projection formula for integration along fibers~\cite[Prop.~IX, I.7.13]{GHV}, we have
\begin{equation}\label{eq:int1}
{u}_{*} ( ({u}^{*} i^{*}\eta) \wedge \sigma) = i^{*}\eta \wedge {u}_{*} \sigma.
\end{equation}
Let $\zeta = \eta - {u}^{*}i^{*}\eta$. Then we have
\begin{equation}\label{eq:zeta1}
{u}_{*}(\eta\wedge \sigma) - i^*\eta\wedge {u}_{*}\sigma = {u}_{*} (\zeta \wedge \sigma).
\end{equation}

Let $f : U \times [0,1] \to U$ be a smooth fiberwise contraction of $U$ to $N$; $f$ is a smooth map such that, denoting $f_{t}(x)=f(x,t)$,
\begin{itemize}
\item $f_{1} = \id_{U}$, $f_{0} = {u}$ and
\item $f_{t}$ preserves $\F$ for each $t \in [0,1]$.
\end{itemize}
Here, by the last lemma, we have
\[
\zeta = \eta - {u}^{*}i^{*}\eta =f_{1}^{*}\eta - f_{0}^{*}\eta=j_{1}^{*}f^{*}\eta - j_{0}^{*}f^{*}\eta=d\mathcal{I} f^{*}\eta + \mathcal{I} d f^{*}\eta.
\]
By taking the wedge product with $\sigma$, we obtain
\[
\zeta \wedge \sigma = (d\mathcal{I} f^{*}\eta) \wedge \sigma + (\mathcal{I} d f^{*}\eta) \wedge \sigma.
\]
By the projection formula for $\mathcal{I}$, we have
\[
(\mathcal{I} d f^{*}\eta) \wedge \sigma = \mathcal{I}\big( (f^{*}d\eta) \wedge \pr_{1}^{*} \sigma \big),
\]
where $\pr_{1} : U \times [0,1] \to U$ is the first projection. Here,
\[
\iota_{\frac{\partial}{\partial t}} \big( (f^{*}d\eta) \wedge \pr_{1}^{*} \sigma \big)= \big( \iota_{\frac{\partial}{\partial t}} (f^{*}d\eta) \big) \wedge \pr_{1}^{*} \sigma = \big( f^{*} \iota_{f_{*}\frac{\partial}{\partial t}} d\eta \big) \wedge \pr_{1}^{*} \sigma
\]
Since $\eta$ is relatively closed and $f_{t}$ preserves $\F$, it follows that the right hand side of the last equation is zero, which implies that $(\mathcal{I} d f^{*}\eta) \wedge \sigma=0$. Hence we obtain
\[
\zeta \wedge \sigma = (d\mathcal{I} f^{*}\eta) \wedge \sigma = d (\mathcal{I} (f^{*}\eta) \wedge \sigma).
\]
With $\theta=u_*(\mathcal{I} (f^{*}\eta) \wedge \sigma)$, Equation \eqref{eq:projformula} follows from Equation \eqref{eq:zeta1}. We see that $\theta$ is compactly supported in the following way. If $\eta_x=0$ for a $x\in N$, then $f^*(\eta)$ is zero on every point in $f^{-1}(x)$. Therefore $\supp \theta$ is contained in the intersection of $N$, which is closed, and $\supp \eta$, which is compact. Therefore $\supp \theta$ is compact.
\end{proof}

As a consequence, the next proposition follows.
\begin{prop}\label{prop:fiberintegration}
For a compactly supported, relatively closed $p$-form $\eta$ on $(M,\F)$, we have
$$
\int_{(\F,\eta)}\sigma=\int_{(\F|N,i^*\eta)}{u}_{*}\sigma
$$
for any $\sigma\in \Omega^{q}_{cv}(U,\F)$.
\end{prop}
\begin{proof}
Since $\supp (\eta \wedge \sigma)$ is contained in $U$, we have
$$
\int_{(\F,\eta)}\sigma\stackrel{\rm def}{=}\int_M\eta\wedge \sigma=\int_N{u}_{*}(\eta\wedge \sigma)=\int_N i^*\eta\wedge {u}_{*}\sigma\stackrel{\rm def}{=}\int_{(\F|N,i^*\eta)}{u}_{*}\sigma.
$$
for any $\sigma\in \Omega_{cv}^{q}(U,\F)$, where the third equation is due to the previous lemma.
\end{proof}

\begin{rem}
The last proposition corresponds to \cite[Lemma 5.4]{Toeben} with $c_i=1$, $N=C_i$ and the transverse integration operator defined in a different way.
\end{rem}

\subsection{Proof of the localization formula (Theorem~\ref{thm:loc})}
We will follow the original argument of \cite{AtiyahBott} adapted to our foliated setting. Let $(M,\F)$ be an oriented manifold with a transversely complete oriented Killing foliation with structural algebra $\mfa$. Let $C$ be the union of all closed leaves in $(M,\F)$. Here $C$ has only finitely many connected components $C_{1}$, $\ldots$, $C_{N}$ by the compactness of the space of leaf closures of $(M,\F)$. Since $(M,\F)$ is transversely complete, by using the normal exponential map of $C_{k}$, we can construct a foliated tubular neighborhood $\rho_{k} : (U_{k},\F) \to (C_{k},\F)$ of $C_{k}$ in $(M,\F)$ so that $U_{1}$, $\ldots$, $U_{N}$ are mutually disjoint. Let $r_{k}$ be the codimension of $C_{k}$ in $M$. According to \cite[Corollary 4.8]{Toeben}, the chain map $\rho_*: \bigoplus \Omega^{r_{k}+\bullet}_{cv}(U_{k},\F)\to \bigoplus \Omega^{\bullet}(C_{k},\F)$ induces an isomorphism in basic cohomology. The {\em basic Thom homomorphism} $i_{*} : H(C, \F) \to H(M,\F)$ is obtained by concatenating $(\rho_{*})^{-1} : H(C, \F) \to H_{cv}(U, \F)$ with the inclusion $H_{cv}(U,\F)\to H(M,\F)$. This map has an equivariant extension,
\[
i_{*} : H_\mfa(C,\F) \to  H_\mfa(M,\F)
\]
which is called the equivariant basic Thom homomorphism (see Definition~\ref{defn:eqThomhom}). The following proposition follows directly from Lemma~\ref{lem:A5} and Equation~\eqref{eq:thomeuler}.
\begin{prop}\label{prop:thom}
We have
\[
i^*i_*1=e_\mfa(\nu C,\F),
\]
where we understand $e_\mfa(\nu C,\F)\in H(C,\F)=\bigoplus H(C_{k},\F)$ as the direct sum of the basic Euler classes $e_\mfa(\nu C_{k},\F)$ of the normal bundles of $(C_{k},\F)$ (see Definition~\ref{defn:eqbasicEulerform}).
\end{prop}
The transverse integration operators for a transversely complete orientable Killing foliation $\F$ on an oriented manifold $M$ with respect to a compactly supported relatively closed form $\eta$ extend equivariantly to $\int_{(\F,\eta)}:H_\mfa(M,\F)\to S(\mfa^*)$ (see \cite[Section 7.2]{Toeben} for details).

The equivariant basic Euler class $e_\mfa(\nu C,\F)$ is not a zero-divisor in $H_\mfa(C,\F)$; indeed, its restriction to a leaf of $C$ is the product of weights of the $\mfa$-action as in the classical case of manifolds with Lie group actions (see Proposition \ref{prop:Eulerclassweights}). There is a basic version of the classical Borel localization theorem for the transverse $\mfa$-action of Killing foliations, in which the role of the fixed point set of a torus action is played by the closed leaves of $\F$, see \cite[Theorem 5.2]{GT2010}. It states that the restriction map $i^*:\widehat H_\mfa(M,\F)\to \widehat H_\mfa(C,\F)$ is an isomorphism on the level of localized modules; recall that $\widehat H_\mfa(M,\F)=Q(\mfa^*)\otimes_{S(\mfa^*)} H_\mfa(M,\F)$, where $Q(\mfa^*)$ is the field of fractions of $S(\mfa^*)$. Therefore, by Proposition~\ref{prop:thom}, $i_* : \widehat H_\mfa(C,\F) \to  \widehat H_\mfa(M,\F)$ is an isomorphism with inverse
$$
S=\sum_{k}\frac{i^*_{C_{k}}}{e_\mfa(\nu C_{k},\F)}.
$$
Thus, for any $\sigma \in H_\mfa(M,\F)$,
we have
\begin{equation}\label{eq:sigma1}
\sigma=i_*S\sigma=\sum_{k}\frac{i_*^{C_{k}}i^*_{C_{k}}\sigma}{e_\mfa(\nu C_{k},\F)}.
\end{equation}
after localization. By localizing $\int_{(\F,\eta)}:H_\mfa(M,\F)\to S(\mfa^*)$, one obtains a map $\int_{(\F,\eta)}:\widehat{H}_\mfa(M,\F)\to Q(\mfa^*)$.
Applying this map to both sides of \eqref{eq:sigma1} and then using Proposition \ref{prop:fiberintegration} for the right side on the the level of equivariant basic cohomology proves Theorem \ref{thm:loc}.

\section{Localization on K-contact manifolds} \label{sec:K-contact}

\subsection{K-contact manifolds}\label{sec:K-contact1}

Let $M$ be a $(2n+1)$-dimensional compact manifold with a contact $1$-form $\eta$. The conditions $\eta(\xi)=1$ and $\iota_\xi d\eta=0$ determine the \emph{Reeb vector field} $\xi$. The two-form $d\eta$ gives the contact structure $\caD:=\ker \eta$ the structure of a symplectic vector bundle. We assume that we are given an almost complex structure $J$ on $\caD$ which is compatible with $d\eta$ in the sense that $d\eta(JX,JY)=d\eta(X,Y)$ for all $X, Y$ and $d\eta(X,JX)>0$ for all $X\neq 0$, and extend $J$ to an endomorphism of $TM$ by setting $J(\xi)=0$. Then, we equip $M$ with the associated Riemannian metric
\begin{equation}\label{eq:associatedmetric}
g = \frac12 d\eta\circ (1\otimes J)+\eta\otimes \eta.
\end{equation}
The tuple $(\xi,\eta,J,g)$ is called a \emph{contact metric structure} (see \cite[Definition 6.4.4]{BoyerGalicki}).
Let $\F$ be the one-dimensional foliation on $M$ defined by the Reeb vector field $\xi$. We are interested in the case when $\F$ is a Riemannian foliation. By \cite[Proposition 6.4.8]{BoyerGalicki}, this condition is equivalent to $M$ being a K-contact manifold:
\begin{defn}
The contact metric structure $(\xi,\eta,J,g)$ on $M$ is called \emph{K-contact} if $\xi$ is a Killing vector field with respect to the Riemannian metric $g$. In this case, $M$ is called a \emph{K-contact manifold}.
\end{defn}

Yamazaki~\cite[Proposition~2.1]{Yamazaki} showed that a contact manifold admits a K-contact structure if the Reeb flow is a Riemannian foliation. In the rest of Section \ref{sec:K-contact} we will consider only compact $K$-contact manifolds.

The most important examples of K-contact manifolds are Sasakian manifolds:
\begin{defn}
If the contact metric structure $(\xi,\eta,J,g)$ on $M$ is K-contact and the CR structure $(\caD,J)$ is integrable, then $M$ is called a \emph{Sasakian manifold}.
\end{defn}

\subsection{Volume of K-contact manifolds} \label{subsec:volume}
As the top form $\eta\wedge (d\eta)^n$ is nowhere vanishing, we can use it to fix an orientation on $M$. Then the Riemannian volume form on $M$ of the metric in \eqref{eq:associatedmetric} is
\begin{equation}\label{eq:kcontvol}
\eta \wedge \frac{(d\eta)^n}{2^nn!}.
\end{equation}
\begin{rem}\label{rem:convention}
Note that \cite{BoyerGalicki} and \cite{MSY} use slightly different conventions. For example, the transverse metric is defined to be $d\eta$ instead of $\frac12 d\eta$ in \cite[Eq.~6.4.2]{BoyerGalicki}. Consequently, the Sasakian volume form in \cite{BoyerGalicki} reads $\eta \wedge \frac{(d\eta)^n}{n!}$. Here we follow the convention of \cite{MSY}.
\end{rem}
Then we observe the following:
\begin{lemma}\label{lem:vol}
The differential form $d\eta$ is basic with respect to the characteristic foliation $\F$, i.e., $\eta$ is relatively closed on $(M,\F)$. Thus the volume of a compact K-contact manifold $(M,\xi,J,g)$ is given by
\[
\Vol(M,g)=\frac{1}{2^nn!}\int_M \eta\wedge (d\eta)^n = \frac{1}{2^nn!} \int_{(\F,\eta)} (d\eta)^n.
\]
\end{lemma}

\subsection{Deformation of Reeb vector fields}\label{sec:deformaiton}

It is well-known that one can use symmetries of $K$-contact manifolds to deform a given $K$-contact structure to new ones, with the same CR structure, but different Reeb vector field. Let $G$ be a Lie group with Lie algebra $\mfg$. Consider a compact K-contact manifold $(M,\xi,\eta,J,g)$ with a $G$-action which leaves the K-contact structure invariant, and define
\[
\cS :=\{X\in \mfg\mid \eta(X^\#)>0\},
\]
where $X^{\#}$ is the fundamental vector field associated to $X$. Then, for any $X \in \cS$, we have a K-contact structure such that
\begin{itemize}
\item the CR structure is equal to $(J,\ker \eta)$.
\item the Reeb vector field is equal to $X^{\#}$.
\item the contact form is equal to $\frac{\eta}{\eta(X^\#)}$.
\end{itemize}
The metric is determined by the formula~\eqref{eq:associatedmetric}.
\begin{rem}
As the CR structure does not change during the deformation, the deformed structure is Sasakian if and only if the original one was Sasakian.
\end{rem}
This construction for Sasakian manifolds goes back to Takahashi \cite{Takahashi} in the case where $G$ is a torus. It is now known as a \emph{deformation of type I} (see \cite[Section 8.2.3]{BoyerGalicki}). For K-contact manifolds, this type of construction was considered in \cite[Proposition~1]{BanyagaRukimbira} and \cite[Lemma~2.5]{Nozawa2} (see also~\cite[Lemma~2.7]{GNT}).

Since the contact form is determined by the contact structure and the Reeb vector field:
\begin{prop}
We have a map
\[
\cS \longrightarrow \RR_{>0}
\]
sending $X\in \mfg$ to the volume of the deformed $K$-contact structure with Reeb vector field $X^\#$.
\end{prop}
This map has been investigated by Martelli-Sparks-Yau~\cite{MSY} in order to find Sasaki-Einstein metrics on a given manifold.

Note that for any $X\in \cS$ we have $[X^{\#},\xi] = 0$. The Killing vector fields $X^{\#}$ and $\xi$ thus arise as fundamental vector fields of an action of a torus leaving invariant the $K$-contact structure. We will hence restrict our attention to the case when $G$ is a torus.

\begin{rem}
For a compact K-contact manifold the closure $T$ of the Reeb flow in the isometry group is a torus.  If the K-contact manifold admits a Reeb orbit which is not closed, then $\dim T \geq 2$. Hence, the Reeb vector field of such K-contact manifolds can be deformed in the way described in this section.
\end{rem}

\subsection{The canonical transverse action on a K-contact manifold}\label{sec:tractiononsas}

Let $M$ be a compact K-contact manifold. Since the Reeb flow of $\eta$ preserves $g$, by~\cite[Th\'{e}or\`{e}me~A]{MolSer}, the orbit foliation $\F$ of the Reeb flow is Killing. The structural Killing algebra $\mfa$ is easily identified as mentioned in \cite[Example~4.3]{GT2010}: Consider the closure $T$ of the flow of the Reeb vector field $\xi$ inside the isometry group $\operatorname{Isom}(M,g)$. As a connected Abelian compact Lie group, it is isomorphic to a torus. Let $b\in \mft$ be the element corresponding to the Reeb vector field, i.e., for which $b^\#=\xi$. Then, the structural Killing algebra of $\F$ is isomorphic to the Lie algebra $\mft/\RR b$, where $\mft$ is the Lie algebra of $T$, which naturally admits an isometric transverse action
\begin{equation}\label{eq:traction}
\mft/\RR b \longrightarrow l(M,\F)
\end{equation}
on $(M,\F)$. Below, we will identify $\mfa=\mft/\RR b$.

We will use our ABBV type formula (Theorem \ref{thm:loc}) to calculate the volume of the K-contact manifold $M$. To apply the theorem directly, those cases are most relevant for us in which not all leaves of the Riemannian foliation $\F$ are closed, i.e., irregular K-contact structures.

\subsection{Equivariant extension of $d\eta$}\label{sec:extension}
Let $M$ be a compact K-contact manifold. In order to apply Theorem \ref{thm:loc}, we need to extend the basic cohomology class $[(d\eta)^n]\in H^{2n}(M,\F)$ to an equivariant basic cohomology class in $H^{2n}_\mfa(M,\F)$. We will do so by defining an explicit equivariant extension of the two-form $d\eta$ to a closed equivariant basic form. Choose an embedding $E:\mfa\to\mft$ such that its concatenation with $\mft\to \mft/\RR b=\mfa$ is the identity. Then we can define an equivariant basic form $\omega$ by
\begin{equation}\label{eq:extensionsasaki}
\omega(X)=d\eta-\eta(E(X)^{\#}),
\end{equation}
where $\omega \in \Omega_{\mfg}(M,\F)$ is regarded as a $\mfg$-equivariant polynomial map $\mfg\to \Omega(M,\F)$.

\begin{lemma}\label{lem:equivariantext}
$\omega$ is equivariantly closed and satisfies $\omega(0)=d\eta$; namely $[\omega]$ extends $[d\eta]$ to equivariant basic cohomology.
\end{lemma}

\begin{proof}
Clearly we have $\omega(0)=d\eta$.
We see that $\omega$ is  equivariantly closed by the following calculation:
$$
(d_\mfa\omega)(X)=d(\omega(X))-\iota_{X^{\#}}\omega(X)=-d(\eta(E(X)^{\#}))-\iota_{X^{\#}}d\eta=-L_{E(X)^{\#}}\eta=0.
$$
\end{proof}
One can thus apply the Localization Theorem \ref{thm:loc} to  $\frac{1}{2^nn!}\omega^n$ in order to calculate the volume of $M$.

\subsection{Volume localization}\label{sec:volloc}
For simplicity, we restrict in this section to the case of a K-contact manifold with only finitely many closed Reeb orbits. Note that in this case, if $T$ is the torus given by the closure of the Reeb field, and ${\mathcal S}\subset \mft$ the set of admissible Reeb fields that can be obtained by a deformation of type I, then a dense subset of vector fields in ${\mathcal S}$ have only isolated closed flow lines. In the examples in Section \ref{sec:ex} we will apply the formula below to Reeb fields in this dense subset, and obtain by continuity the full volume map ${\mathcal S}\to \RR_{>0}$. Note that, since the isotropy action is symplectic at any point on $M$, its weights are well defined by using the decomposition into irreducible representations (see~\cite[Lemma 3.11]{Lerman}).

\begin{thm}\label{thm:volumelocalization} Let $M$ be a $(2n+1)$-dimensional compact K-contact manifold with only finitely many closed Reeb orbits $L_1$, $\ldots$, $L_N$. Denote the weights of the transverse isotropy $\mfa$-representation at $L_{k}$ by $\{\alpha^k_j\}_{j=1}^{n} \subset \mfa^*$ for $k=1$, $\ldots$, $N$. Then the volume of $M$ is given by
\begin{equation}\label{eq:volumelocalization}
\Vol(M,g) = \frac{\pi^n}{n!} \sum_{k=1}^N l_k\cdot \frac{\eta|_{L_k}(v^{\#})^n}{\prod_j \alpha^k_j(v+\RR b)},
\end{equation}
where $l_k=\int_{L_k} \eta$ is the length of the closed Reeb orbit $L_k$, and the fractions on the right hand side are considered as rational functions in the variable $v\in \mft$. In particular, the right hand side of \eqref{eq:volumelocalization} is independent of $v$.
\end{thm}
\begin{proof}
Let $v\in \mft$, and define $Y=v+\RR b\in \mfa$. Then we can choose an embedding $E:\mfa\to \mft$ that sends $Y$ to $v$ and is right-inverse to the projection $\mft\to \mft/\RR b=\mfa$. Denote the corresponding equivariant extension of $d\eta$ defined in \eqref{eq:extensionsasaki} by $\omega$. In particular, $\omega(Y) = d\eta - \eta(v^{\#})$. By Corollary \ref{cor:isolatedclosedleaves} and Lemma~\ref{lem:equivariantext} we have
\[
\Vol(\xi) = \frac{1}{2^nn!}\int_{(\F,\eta)} \omega^n(Y) =
\frac{(-\pi)^n}{n!} \sum_{i=1}^N \int_{L_{k}}\eta\wedge \left(\frac{i^*_{L_{k}}\omega^n(Y)}{\prod_{j} \alpha^{k}_j(Y)}\right).
\]
As $L_{k}$ is one-dimensional, the nominator simplifies to $(-1)^n \left.\eta(v^{\#})^n\right|_{L_{k}}$.
\end{proof}

\subsection{K-contact Duistermaat-Heckman}

Before turning to the application of Theorem \ref{thm:volumelocalization} to concrete situations, we derive a K-contact version of the Duister\-maat-Heckman formula. Atiyah-Bott \cite{AtiyahBott} and Berline-Vergne \cite{BerlineVergne2} observed that the Duistermaat-Heckman formula \cite{DuistermaatHeckman} is an easy corollary of the localization formula. In this section, we proceed analogous to their argument.

Let $M$ be a compact K-contact manifold of dimension $2n+1$, with K-contact metric structure $(\xi,\eta,J,g)$ and denote the Riemannian volume form of $g$ as in \eqref{eq:kcontvol} by
\[
\beta=\frac{1}{2^n n!} \, \eta\wedge (d\eta)^n.
\]

\begin{thm} Let $M$ be a compact K-contact manifold such that the action of the closure $T$ of the Reeb flow admits only finitely many $S^{1}$-orbits $L_1$, $\ldots$, $L_N$.  Denote the weights of the transverse isotropy $\mfa$-representation at $L_{k}$ by $\{\alpha^{k}_j\}_{j=1}^{n} \subset \mfa^*$ for $k=1$, $\ldots$, $N$. Then we have for all $v\in \mft$
\[
\int_M e^{\eta(v^\#)} \beta = \pi^n\sum_{k=1}^N l_{k} \cdot \frac{e^{\eta(v^\#)}(L_{k})}{\prod_j \alpha^{k}_j(v+\RR b)},
\]
where $l_{k}=\int_{L_{k}}\eta$ is the length of the closed Reeb orbit $L_{k}$.
\end{thm}
\begin{proof} First, consider the case where $M$ admits only finitely many closed Reeb orbits $L_1$, $\ldots$, $L_N$. Consider the equivariant extension $\omega$ of $d\eta$ as in \eqref{eq:extensionsasaki}: $\omega(X)=d\eta - \eta(E(X)^\#)$; we choose $E:\mfa\to \mft$ such that $E(X)=-v$ for some $X\in \mfa$. We would like to apply Theorem \ref{thm:loc} to  $e^\omega$, but as this is not an equivariant basic differential form in our sense, we apply it separately to each summand in the power series:
\begin{multline*}
\int_M e^{\eta(v^\#)}\beta = \frac{1}{2^n} \sum_{s=0}^\infty  \int_M \eta\wedge \frac{\omega^s(X)}{s!} =\frac{1}{2^n} \sum_{s=0}^\infty  \int_{(\F,\eta)} \frac{\omega^s(X)}{s!} \\
= (-\pi)^n \sum_{s=0}^\infty \sum_{k=1}^N l_{k} \frac{(-\eta(E(X)^\#))^s(L_{k})}{s!\prod_j \alpha_j^{k}(X)} = \pi^n \sum_{k=1}^N l_{k} \frac{e^{\eta(v^\#)}(L_{k})}{\prod_j \alpha^{k}_j(v+\RR b)}
\end{multline*}
The general case follows from the first case like in the argument in the last paragraph of the proof of Theorem~\ref{thm:volumelocalization}.
\end{proof}

\section{Examples}\label{sec:ex}

\subsection{Deformation of standard Sasakian spheres}\label{subseq:Deformation}

We will compute the volume of deformations of the standard Sasakian structure on $S^{2n+1}$ (see \cite[Examples 7.1.5 and 7.1.12]{BoyerGalicki}). The Riemannian metric $g$ is the standard one, with sectional curvature $1$, inherited from the flat metric on $\RR^{2n+2}$, and the contact one-form is $\eta =  \sum_{i=0}^n (x_idy_i-y_idx_i)$. Let $\{2\pi e_i\}_{i=0,\ldots,n}$ be the standard basis of $\mft$, i.e. in particular each element is in the integral lattice; the fundamental vector field on $S^{2n+1}$ induced by $e_i$ is
\[
H_i = -y_i \frac{\partial}{\partial x_i} + x_i \frac{\partial}{\partial y_i}.
\]
Then the standard Reeb vector field is $\xi=\sum_i H_i=\sum_i e_i^\#$. For any vector $w=(w_0,\ldots,w_n)$ with $w_i>0$ we consider the deformed Reeb vector field $\xi_w=b_{w}^\#$, where $b_w = \sum_i w_i e_i\in \mft$. The contact form of the deformed Sasakian structure is given by
\[
\eta_w = \frac{\eta}{\sum_{i=0}^n w_i(x_i^2+y_i^2)}.
\]
For generic choice of $w$, the only closed leaves of $\F_w$ are those given by $|z_i|=1$. We denote these by $L_0$, $\ldots$, $L_n$. 

We will now prove the following consequence of Theorem \ref{eq:volumelocalization}, which are also consequences of results of Martelli-Sparks-Yau and Lawrence (see the next section for the statement of their results).

\begin{cor}\label{cor:voldeformedsphere}
We have
\begin{align}
\Vol(S^{2n+1},g_w) & = \frac{2\pi^{n+1}}{n!} \sum_{i=0}^{n} \frac{1}{w_i^{n+1}} \frac{\beta_i^n}{\prod_{j=0,\ldots,n,\, j\neq i} (\frac{\beta_i}{w_i} w_j-\beta_j)} \label{eq:volumes2n+1} \\
& = \frac{2\pi^{n+1}}{n!} \cdot \frac{1}{w_0\cdot\cdots \cdot w_n}. \label{eq:volumedeformedsphere}
\end{align}
\end{cor}

\begin{proof}
We need to calculate $\int_{L_i} \eta_w$, as well as the nominator and the denominator of the right hand side of \eqref{eq:volumelocalization}.

It is easy to see that $\int_{L_i} \eta_w=\frac{2\pi}{w_i}$. The weights $\{\alpha^i_j\}_{j=0, \ldots, \hat i, \ldots, n}$ of the transverse isotropy representation at $L_i$ (see Appendix \ref{sec:eqcharclasses}) are given by the negative of the dual basis of $\{e_k+\RR b_w\}_{k=0, \ldots, \hat i, \ldots n}$. Explicitly, $\alpha^i_j(e_k+\RR b_w) = -\delta_{jk}$ for $j, k\neq i$. This is because $[e_k^\#,\cdot\,]=[H_k,\cdot\,]$, with respect to $\big\{\frac{\partial}{\partial x_j},\frac{\partial}{\partial y_j}\big\}$, has block diagonal shape $\delta_{jk}\left(\begin{smallmatrix} 0 & 1 \\ -1 & 0 \end{smallmatrix}\right)= \left(\begin{smallmatrix} 0 & -\alpha_j(e_k+\RR b_w) \\ \alpha_j(e_k+\RR b_w) & 0 \end{smallmatrix}\right)$. Writing
\[
v=\sum_{i=0}^n \beta_j e_j,
\]
we calculate $v+\RR b_w = \sum_{j\neq i} \beta_j (e_j+\RR b_w) - \frac{\beta_i}{w_i} \sum_{k\neq i} w_k (e_k+\RR b_w)$.  Then the denominator is given by
\begin{equation*}
\prod_{j\neq i} \alpha^i_j(v+\RR b_w)=  \prod_{j\neq i} \left(\frac{\beta_i}{w_i} w_j-\beta_j\right).
\end{equation*}
Because the $\mft$-isotropy subalgebra at $L_i$ is spanned by $\{e_j\}_{j=0, \ldots, \hat i, \ldots, n}$, we have that $v-\frac{\beta_i}{w_i}b_w\in \mft_{L_i}$. The numerator is therefore given by
\[
\eta_w|_{L_i}(v^{\#})^n=\left(\frac{\beta_i}{w_i}\right)^n.
\]
Substituting these equations, \eqref{eq:volumes2n+1} follows from \eqref{eq:volumelocalization}.

As the right hand side of~\eqref{eq:volumes2n+1} is independent of $\beta_j$, we can send $\beta_0$ to $+\infty$ in order to simplify the expression. When doing so, the summands for $i\neq 0$ tend to zero, and the summand for $i=0$ tends to $\frac{1}{w_0\cdot\cdots\cdot w_n}$. Hence, we get~\eqref{eq:volumedeformedsphere}.
\end{proof}

\subsection{Toric Sasakian manifolds}
Let  $(M,\eta,g)$ be a Sasakian manifold of dimension $2n+1$, and denote by $T$ the torus obtained as the closure of the Reeb flow in $\Isom(M)$. By a result of Rukimbira~\cite[Corollary 1]{Rukimbira4}, one has $\dim T \leq n+1$. If $\dim T = n+1$ holds, one calls $(M,\eta,g)$ a {\em toric} Sasakian manifold.

We briefly describe the Delzant-type correspondence between toric Sasakian $(2n+1)$-manifolds and certain cones in a Euclidean space due to Lerman~\cite{Lerman} and Boyer-Galicki~\cite{BoyerGalicki2}. Letting $\mft=\Lie(T)$, the contact moment map is defined by
\[
\begin{array}{cccc}
\Phi : & M & \longrightarrow & \mft^{*} \\
         & x  & \longmapsto     & (X \mapsto \eta(X^{\#})(x) )
\end{array}
\]
We define a cone in $\mft^*$ by $\Delta=\RR_{\geq 0}\cdot \Phi(M)$, and put $\Delta_t = t\cdot \Phi(M)$. Note that $\Phi(M)=\Delta_1 = H\cap \Delta$, where $H = \{ \varphi \in \mft^{*} \mid \varphi(b)= 1 \}$. Denote the element of $\mft$ whose fundamental vector field is the Reeb field by $b\in \mft$. Lerman showed in \cite[Theorem 2.18]{Lerman} that $\Delta$ is a good rational polyhedral cone (because $M$ is a contact toric manifold of Reeb type), i.e., there exists a minimal set of primitive vectors $\{v_{i}\}_{i \in I} \subset \mft_{\ZZ}$ in the integral lattice $\mft_{\ZZ} = \ker (\exp : \mft \to T)$ such that
\begin{enumerate}
\item $\Delta = \{ \varphi \in \mft^{*} \mid \varphi(v_{i}) \leq 0 \}$, and
\item for any face of $\Delta$ of the form $\Delta\cap \bigcap_{j=1}^k \{\varphi\in \mft^*\mid \varphi(v_{i_j})=0\}\neq \{0\}$, we have
\[
\left(\bigoplus_{j=1}^k \RR \cdot v_{i_j}\right) \cap \mft_\ZZ = \bigoplus_{j=1}^k \ZZ\cdot v_{i_j}.
\]
and $v_{i_1},\ldots,v_{i_k}$ are linearly independent over $\ZZ$.
\end{enumerate}
In particular, for any closed Reeb orbit $L$, $\Phi(L)$ is a vertex of $\Delta_1$, and thus there exist precisely $n$ of the vectors $v_i$, denoted $v_1^L,\ldots,v_n^L$, such that $\Phi(L)(v_i^L)=0$. We fix a determinant $\det\in \wedge^n\mft^*$ such that a chosen integer basis of $\mft$ is sent to $1$ by the determinant. We then assume that the vectors $v_i^L$ are ordered in such a way that $\det(b,v_1^L,\ldots,v_n^L)> 0$.

Using our localization formula we can now prove:
\begin{thm}\label{thm:localizationvolumeSasakiantoric}
We have
\begin{multline*}
\Vol (M,\eta_{b}) = \frac{1}{2^nn!} \sum_{L} \frac{1}{\det (b, v_{1}^{L}, \ldots, v_{n}^{L})}\cdot \\ \frac{\det (v, v_{1}^{L}, \ldots, v_{n}^{L})^{n}}{\prod_{i=1}^{n}  \det (b, v_{1}^{L}, \ldots, v_{i-1}^L, v, v_{i+1}^L, \cdots, v_{n}^{L})}.
\end{multline*}
\end{thm}
\begin{proof}
We first show~
\begin{equation}\label{eq:length}
\int_{L} \eta_{b} = \frac{1}{\det (b, v_{1}^{L}, \ldots, v_{n}^{L})}.
\end{equation}
Let $T_{L}$ be the isotropy group of the $\mft$-action at $L$ and $\mft_{L}$ its Lie algebra. We identify $L$ with $T/T_{L}$ and let $p_{L} : T \to T/T_{L}$ be the projection. We identify $\Lie(T/T_{L})$ with $\RR$ so that $\ker \exp$ is identified with $\ZZ$ and such that $b+\mft_L$ is sent to a positive number. Since $T_{L}$ is connected by \cite[Lemma 3.13]{Lerman}, the Reeb vector field $\xi_b$ on $T/T_{L}$ is $(p_{L})_{*}b$. By $\eta_{b}(\xi_b)=1$, we get
\begin{equation}
\int_{L} \eta_{b} = \frac{1}{(p_{L})_{*}b}.
\end{equation}
On the other hand, we now argue that $(p_{L})_{*}b = \det (b, v_{1}^{L}, \ldots, v_{n}^{L})$: for this we first observe that $\mft_L$ is spanned by the $v_i^L$ (see, e.g., \cite[Proof of Lemma 6.4]{Lerman}: there it is argued that if $F$ is the face of $\Delta$ containing $\Phi(p)$ in its interior, then the real span of $F$ is equal to the annihilator $\mft_x^\circ$). That shows that the linear forms $(p_L)_*$ and $\det(\cdot,v_1^L,\ldots,v_n^L)$ have the same kernel. To see that they are actually equal we choose, using the fact that $\Delta$ is a good rational polyhedral cone, a vector $v_0^L$ completing the $v_i^L$ to an integer basis of $\mft_\ZZ$, such that  $\det(v_0^L,v_1^L,\ldots,v_n^L)=1$. With the help of this basis we can identify $\mft$ with $\RR^{n+1}$ in such a way that $\mft_\ZZ$ is sent to $\ZZ^{n+1}$ and $\mft_L$ to $\RR^{n}$. Because $T_L$ is connected \cite[Lemma 3.13]{Lerman}, this implies that $(p_L)_*(v_0^L)=1$. This implies~\eqref{eq:length}.

Denoting by $p_{\mfa} : \mft \to \mfa$ the canonical projection, $\{p_{\mfa}(v_{i}^{L})\}_{i=1}^{n}$ is a basis of $\mfa$. We will show next that the transverse weights $w_i^L$ of the $\mfa$-action on $\nu L$ are given by $2\pi$ times the corresponding dual basis; i.e., $w_{i}^{L}\in \mfa^{*}$ so that $w_{i}^{L}(p_{\mfa}(v_{j}^{L})) = 2\pi\delta_{ij}$. Via the natural isomorphism $\mft_L^*\cong \mfa^*$ induced by $\mft_L\subset \mft\to \mfa$ the transverse weights are identified with the ordinary weights of the symplectic slice representation at $L$. In the proof of \cite[Lemma 6.4]{Lerman} it is shown that the momentum image of this symplectic slice representation is the cone given as the intersection of the half-spaces $\{\varphi\in \mft_L^*\mid \varphi(v_i^L)\leq 0\}$. Thus, the edges of this cone are spanned by the negative of the dual basis of the $v_i^L$. On the other hand it is known that (with our sign convention) the negatives of the weights of the symplectic slice r
 epresentation also span the edges of the cone. Because it is shown in the proof of \cite[Lemma 6.4]{Lerman} that the $v_i^L$ form a basis of the integer lattice (and hence also the dual basis), the weights are necessarily given by $2\pi$ times the corresponding dual basis.

In other words we have shown that
\begin{equation}\label{eq:toricweightsexplicitly}
w_{i}^{L}(v+\RR b) = \frac{2\pi\det (b, v_{1}^{L}, \ldots, v_{i-1}^{L}, v, v_{i+1}^L, \ldots, v_{n}^{L})}{\det (b, v_{1}^{L}, \ldots, v_{n}^{L})}
\end{equation}
for all $v\in \mft$, as can be easily checked by inserting $v=v_j^L$.

Inserting $b$ and $v_i^L$ on both sides of the following equality shows
\begin{equation}\label{eq:momtoric}
 \eta_{b}(v^{\#})|_{L} = \frac{\det (v, v_{1}^{L}, \ldots, v_{n}^{L})}{\det (b, v_{1}^{L}, \ldots, v_{n}^{L})}\;.
\end{equation}
By substituting~\eqref{eq:length},~\eqref{eq:toricweightsexplicitly} and~\eqref{eq:momtoric} into the localization formula \eqref{eq:volumelocalization}, we get
\begin{align*}
\Vol (M,\eta_{b}) &= \frac{\pi^n}{n!} \sum_{L} l_L \cdot \frac{\eta|_{L}(v^{\#})^n}{\prod_i w_{i}^{L}(v+\RR b)}\\
&=\frac{1}{2^nn!} \sum_{L} \frac{1}{\det (b, v_{1}^{L}, \ldots, v_{n}^{L})}\cdot \\ &\qquad\qquad\qquad\qquad \frac{\det (v, v_{1}^{L}, \ldots, v_{n}^{L})^{n}}{\prod_{i=1}^{n}  \det (b, v_{1}^{L}, \ldots, v_{i-1}^L, v, v_{i+1}^L, \ldots, v_{n}^{L})}\;.\label{eq:volumeetab}
\end{align*}
\end{proof}

\begin{rem}
We can eliminate the variable $v$ from the formula in Theorem~\eqref{thm:localizationvolumeSasakiantoric} by cutting $\Delta_{1}$ into a finite union of $n$-simplices and taking a limit of the term corresponding to each simplex in a similar way as in the case of spheres in the last section.
\end{rem}
\begin{rem} Of course Corollary \ref{cor:voldeformedsphere} is a special case of Theorem \ref{thm:localizationvolumeSasakiantoric}. For a toric deformation of the standard Sasakian sphere the cone $\Delta$ is, using the notation of Section \ref{subseq:Deformation}, spanned by the basis dual to the $e_i$. The primitive vectors $v_i$ are given by $v_i=-2\pi e_i$. For the closed Reeb orbit $L_j$ the vectors $v_i^{L_j}$ are precisely those $v_i$ with $i\neq j$. The simplest way to compute the right hand side of Theorem \ref{thm:localizationvolumeSasakiantoric} is then to expand the vector $v$ in the basis $e_i$, and send one  coefficient to infinity. Note that the apparently missing factors of $2\pi$, when comparing with Corollary \ref{cor:voldeformedsphere}, are accounted for by our choice of the determinant.
\end{rem}

In the remainder of this section we will show that this formula for the volume of a toric Sasakian manifold can also be proven differently, by combining the fact that the volume of $M$ is closely related to the volume of a truncated momentum cone, as shown by Martelli-Sparks-Yau, and a formula for the volume of a simple polytope due to Lawrence.

Let $C(M)\cong \RR^+\times M$ be the K\"ahler cone of $(M,g,\eta)$ with metric $g'=dr^2\oplus r^2 g$ and symplectic form $\omega$ with $g'(X,Y)=\omega(X,JY)$. For $\eta':=\frac{1}{r}\iota_{\partial/\partial r}\omega$ we have $d(\frac{1}{2}(r^2\eta'))=\omega$, because $L_{r\partial/\partial r}\omega=2\omega$ (as in \cite{MSY2}, Eqs.\ (2.3) and (2.4)), and $\eta'|_{r=1}=\eta$. We define the moment map $\mu:C(M)\to\mft^*$ by $\mu^X=\frac{r^2}{2} \eta'(X^\#)$; note that we use the convention $d\mu^X=-\iota_{X^\#}\omega$. Observe that $\mu(M) = \Delta_{1/2}=\frac12 \Phi(M)$. Let $y_i:=\mu^{e_i}=\frac{r^2}{2}\eta'(e_i^\#)$ so that $dy_i=-\iota_{e_i^\#}\omega$.

Let $\{e_i\}_{i=0,\ldots, n}$ be a basis of $\mft$ such that $\{2\pi e_i\}_{i=0,\ldots, n}$ is an integer basis of the integral lattice $\mft_\ZZ$. Let $\phi=(\phi_0,\ldots,\phi_n):\mft\to \RR^{2n+1}/2\pi\ZZ^{2n+1}$ denote the coordinates with respect to the basis. We first understand the $d\phi_i$ as forms on the principal orbits dual to the fundamental fields. We then extend these leafwise forms to forms, also denoted by $d\phi_i$, on the regular part of $C(M)$ such that its kernel contains the span of the $Je_i^\#$.  Then $\omega=\sum_{i=0}^n dy_i\wedge d\phi_i$ on the regular part. The volume with respect to $e_0\wedge \ldots\wedge e_n$ on $\mft^*$ is $(2\pi)^{n+1}$.

We can regard $b \in \mft$ as a $1$-form on $\mft^{*}$. Then $\Omega_{H}$ is the unique translation invariant $n$-form on $\mft^{*}$ such that $b \wedge \Omega_{H}=e_0\wedge\ldots\wedge e_n$. The volume of the dual integral lattice with respect to the latter volume form is $\frac{1}{(2\pi)^{n+1}}$. Let $\Vol_{H}(\Delta_{1/2})$ (resp.\ $\Vol_{H}(\Delta_{1})$) be the volume of $\Delta_{1/2}$ (resp.\ $\Vol_{H}(\Delta_{1})$) with respect to the $n$-form $\Omega_{H}$.
\begin{prop}[{\cite[Eq.\ 2.74 and 2.86]{MSY2}}]\label{prop:MSY}
\begin{equation}\label{eq:volMDelta}
\Vol (M,\eta_{b}) = 2\pi^{n+1}\Vol_{H}(\Delta_{1}).
\end{equation}
\end{prop}

\begin{proof}
We fix a metric on $\mft$ such that the $e_i$ are orthonormal, and the dual metric on $\mft^{*}$. Let $\Omega_{0}$ be the volume form on $H$ determined by orthonormal frames. Then we have $\Omega_{0} = \|b\| \Omega_H$. Let $\Vol_{0}(\Delta_{1/2})$ be the volume of $\Delta_{1/2}$ with respect to $\Omega_{0}$. We have $\Vol_{0}(\Delta_{1/2}) = \|b\| \Vol_{H}(\Delta_{1/2})$. From
$$
\frac{1}{n!}\int_{\mu^{-1}(\Delta_{\leq 1/2})}  \omega^n=\int_{\mu^{-1}(\Delta_{\leq 1/2})} |d\phi_0\ldots d\phi_n dy_0\ldots dy_n|=(2\pi)^{n+1}\Vol(\Delta_{\leq 1/2}),
$$
where $\Delta_{\leq 1/2}=\{w\in \Delta\mid w(b)\leq 1/2\}$,
it follows that we have
\[
\Vol (M,\eta_{b})= 2(n+1)\Vol(\mu^{-1}(\Delta_{\leq 1/2}))= 2(n+1)(2\pi)^{n+1}\Vol(\Delta_{\leq 1/2}),
\] see \cite[Eqs.\ (2.72)-(2.74)]{MSY2}. Since $\Delta_{\leq 1/2}$ is a pyramid of height $1/2\|b\|$ with the base $\Delta_{1/2}$ (if we write $b=\sum_i b_i e_i$, then $\Delta_{1/2}=\{\sum a_i e_i^*\mid \sum a_i b_i=\frac12\}$, hence the shortest element is $(\sum b_i e_i^*)/2\|b\|^2$, which has length $1/2\|b\|$.), we have
\[
\Vol (\Delta_{\leq 1/2}) = \frac{1}{2 (n+1)\|b\|} \Vol_{0} \Delta_{1/2} = \frac{1}{2(n+1)} \Vol_H \Delta_{1/2}.
\]
Since we have $2^{n}\Vol_{H}(\Delta_{1/2})=\Vol_{H}(\Delta_{1})$, \eqref{eq:volMDelta} follows.
\end{proof}

The following application of a formula of Lawrence for the volume of a simple polytope gives a method to compute $\Vol_{H}(\Delta_{1})$.
\begin{thm}[{\cite[Theorem on p.~260]{Lawrence}}]\label{thm:Lawrence}
Take $u \in \mft$ and $d \in \RR$ so that the function $f(x) = u(x) + d$ on $\mft^{*}$ is nonconstant on each edge of $\Delta_{1}$. Then
\begin{equation}\label{eq:Lawrence}
\Vol_{H} (\Delta_{1}) = \frac{1}{n!} \sum_{L} \frac{f(\Phi (L))^{n}}{ \delta^{L} \gamma_{1}^{L} \cdot \ldots \cdot \gamma_{n}^{L}},
\end{equation}
where $L$ runs over the set of closed Reeb orbits,  $\gamma_{i}^{L} \in \RR$ is determined by
\[
u = \gamma_{0}^{L} b + \gamma_{1}^{L} v_{1}^{L} + \cdots + \gamma_{n}^{L} v_{n}^{L},
\]
and $\delta^{L} = \det (b,v_{1}^{L},\cdots,v_{n}^{L})$.
\end{thm}
\begin{proof} Lawrence's Theorem applies to simple polytopes
\[
P = \{v\in V\mid \langle v,a_i\rangle \leq b_i\}
\]
in a Euclidean vector space $(V,\langle\cdot,\cdot\rangle)$, where $a_i\in V$, and $b_i\in \RR$.

In order to apply the Theorem of Lawrence, we identify the affine hyperplane $H$ with the vector space $b^\perp=\{v\in \mft^*\mid v(b)=0\}$ using the map
\[
\Psi: H=\{w\in \mft^*\mid w(b)=1\}\longrightarrow b^\perp;\, w\longmapsto w-\frac{1}{\|b\|^2} b^*,
\]
where $v\mapsto v^*$ denotes the identification $\mft\to \mft^*$.

Let $a_i\in b^\perp$ and $b_i\in \RR$ be defined as
\[
a_i:= v_i^* -\frac{\langle v_i,b\rangle}{\|b\|^2}b^*,\qquad b_i = -\frac{\langle v_i,b\rangle}{\|b\|^2},
\]
where $\langle\cdot,\cdot\rangle$ is the chosen inner product on $\mft$. Denoting with the same symbol the dual inner product on $\mft^*$ we then have
\[
\langle \Psi(w),a_i\rangle = \left\langle w-\frac{1}{\|b\|^2} b^*,v_i^*-\frac{\langle v_i,b\rangle}{\|b\|^2}b^*\right\rangle = w(v_i) -\frac{\langle v_i,b\rangle}{\|b\|^2} = w(v_i) + b_i
\]
for $w \in H$, hence $\Psi(\Delta_{1}) = \{v\in b^\perp\mid \langle v,a_i\rangle \leq b_i\}$. Noting that $\Vol_0(\Delta_{1})=\Vol(\Psi(\Delta_{1}))$, where the volume on the right hand side is defined using the inner product $\langle\cdot,\cdot\rangle$, we apply Lawrence's Theorem: we choose a function $f$ on $\mft^*$ of the form $f(v)= \langle u,v\rangle + d$ for some $u\in b^\perp$, and obtain (note that $f(\Psi(w)) = f(w)$)
\[
\Vol_0(\Delta_{1}) = \frac1{n!} \sum_L \frac{f(\Phi(L))^n}{\epsilon^L \gamma_1^L\cdot \ldots\cdot \gamma_n^L},
\]
where $\epsilon^L=|\det(a_1^L,\ldots,a_n^L)|$ and $u = \gamma_1^L a_1^L+\ldots + \gamma_n^La_n^L$. Here, $\det$ is a Euclidean determinant on $b^\perp$, hence $\epsilon^L = \frac{1}{\|b\|}|\det(b,a_1^L,\ldots,a_n^L)|=\frac{1}{\|b\|} \delta^L$. Because $a_i$ and $v_i^*$ differ only by a multiple of $b^*$, the numbers $\gamma_1^L,\ldots,\gamma_n^L$ coincide with those defined in the statement of the theorem. The proof follows from $\Vol_{0}(\Delta_{1}) = \|b\| \Vol_{H}(\Delta_{1})$.
\end{proof}

By combining Proposition \ref{prop:MSY} with Theorem \ref{thm:Lawrence} we obtain a second proof of Theorem \ref{thm:localizationvolumeSasakiantoric}, because $\gamma_i^L = \det(b,v_1^L,\ldots,v_{i-1}^L,u,v_{i+1}^L,\ldots,v_n^L)$. In fact, any two of these results imply the third.

\subsection{The homogeneous Sasakian manifold $\SO(5)/\SO(3)$}

In this section, we calculate the volume of deformations of the homogeneous Sasaki structure on the Stiefel manifold. We use the notation of \cite[Section 8]{GNT}: consider the real cohomology $7$-sphere $\SO(5)/\SO(3)$; here, $\SO(3)$ is embedded in $\SO(5)$ as $I_2\times \SO(3)$, where $I_2$ is the $(2\times 2)$-identity matrix.

$T^3 = T^2\times T^1=(\SO(2)\times \SO(2)\times 1)\times \SO(2)\subset \SO(5)\times \SO(2)$ acts on $\SO(5)/\SO(3)$ via $(g,h)\cdot k\SO(3)=gkh^{-1}\SO(3)$.
We identify the Lie algebra $\mft^3$ of $T^3 = T^2\times T^1\subset \SO(5)\times \SO(2)$ with $\RR^3$; a vector $w\in \RR^3$ corresponds to an element $b_w\in \mft^3$. The Reeb vector field $\xi$ of the homogeneous K-contact structure is the fundamental vector field of $b=b_{(0,0,1)}$. (Note that at $e\SO(3)$ the Reeb vector field $\xi$ is equal to $-E_{12}+\so(3)$, where $E_{ij}$ is the matrix with all entries zero except $1$ at the $ij$-entry and $-1$ at the $ji$-entry.) Let $w=(x, y, z)\in \mft^3$ be such that the corresponding one-parameter subgroup has only finitely many closed orbits in $\SO(5)/ \SO(3)$. The closed orbits then coincide with the one-dimensional $T^3$-orbits, of which there exist exactly $4$ (see \cite[Section 8]{GNT}). As explained in Section~\ref{sec:deformaiton}, the deformed contact form $\eta_w$ with Reeb vector field $\xi_w=b_w^\#$ is given by
\[
\eta_w = \frac{\eta}{\eta(\xi_w)},
\]
where $\eta = \langle b,\cdot\rangle$, and $\langle \cdot,\cdot\rangle$ is induced by a bi-invariant Riemannian metric on $\SO(5)$, normalized so that $\|E_{12}\|=1$. 
We fix a vector $v=(\alpha,\beta,\gamma)\in \mft$. We will show the following by the localization theorem.
\begin{thm}
We have
\begin{equation}
\Vol(M,g_w) =  \frac{2\pi^4}{ 3}\frac{1}{(z^2-y^2)(z^2-x^2)}.
\end{equation}
\end{thm}

\begin{proof}
Let us calculate in detail the relevant data at the closed Reeb orbit through $e\SO(3)$. Along this closed Reeb orbit, $\xi_w$ is exactly $(z-x)b^\#$, so its length calculates as
\[
\int_{\SO(2)\cdot e\SO(3)} \eta_w = \frac{2\pi}{z-x}.
\]
Still along this closed Reeb orbit, we have $\eta(\xi_w) = z-x$; moreover, $\eta(v^\#)_{e\SO(3)}=\gamma-\alpha$. Therefore, the nominator in the localization formula is
\[
\left.\eta_w\right|_{\SO(2)\cdot e\SO(3)}(v^\#)^3 =  \left(\frac{\gamma-\alpha}{z-x}\right)^3.
\]
We have $\mft^3_{eK} = \{(x,y,x)\mid x,y\in \RR\}$, and the projection $\mft^3_{eK}\to \mft^2=\mft^2\oplus\{0\}\subset\mft^3$ onto the first summand is an isomorphism. The natural projection $T_{eK}G/K\to T_{eH}G/H$, where $G/H=\SO(5)/\SO(2)\times \SO(3)$ becomes equivariant with respect to this homomorphism. Considering the induced isomorphism
\begin{equation}\label{eq:hom1}
\mft^2\longrightarrow \mft^3_{eK}\longrightarrow \mft^3\longrightarrow \mft^3/\RR b_w =\mfa,
\end{equation}
where the first map is the inverse of the projection $\mft^3_{eK}\to \mft^2$, we see that the transverse weights in $\mfa^*$ are in one-to-one correspondence to the ordinary weights of the $\mft^2$-isotropy on $\SO(5)/\SO(2)\times \SO(3)$: if $\tau\in (\mft^2)^*$ is a weight, then the induced transverse weight in $\mfa^*$ is given by
\[
(\alpha,\beta,\gamma)+\RR b_w\longmapsto \tau\left(\alpha+\frac{\alpha-\gamma}{z-x}\cdot x, \, \beta+\frac{\alpha-\gamma}{z-x}\cdot y\right).
\]
To see this, observe that the inverse $\mfa\to \mft^2$ of the isomorphism \eqref{eq:hom1} is given by $(\alpha,\beta,\gamma)+\RR\xi\mapsto (\alpha+\frac{\alpha-\gamma}{z-x}\cdot x,\beta+\frac{\alpha-\gamma}{z-x}\cdot y)$. The ordinary weights of the $\mft^2$-isotropy are given by positive roots of $\SO(5)$ which are not roots of $\SO(2)\times \SO(3)$. Denoting the standard basis of $\mft^2$ by $\{e_1,e_2\}$, the weights are $e_1^*$, $e_1^*+e_2^*$ and $e_1^*-e_2^*$, because $e_2^*$ is also a root of $\SO(3)\times \SO(2)$.

Thus, the denominator is given by
\begin{multline*}
\left(\alpha+\frac{\alpha-\gamma}{z-x}\cdot x\right)\cdot \left(\alpha+\frac{\alpha-\gamma}{z-x}\cdot x+\beta+\frac{\alpha-\gamma}{z-x}\cdot y\right) \cdot \\ \left(\alpha+\frac{\alpha-\gamma}{z-x}\cdot x-\left(\beta+\frac{\alpha-\gamma}{z-x}\cdot y\right)\right).
\end{multline*}
We omit the calculations for the other three closed Reeb orbits, which are parame\-trized by the quotient of Weyl groups $W(\SO(5))/W(\SO(2)\times \SO(3))$. The four summands in the localization formula \eqref{eq:volumelocalization} and the final formula for the volume read as follows:
\begin{align*}
&\Vol(M,g_w) = -\frac{2\pi^4}{3!} \cdot\\
&\Bigg[ \frac{(\gamma-\alpha)^3}{(z-x)^4 (\alpha+\frac{\alpha-\gamma}{z-x} \cdot x) (\alpha+\frac{\alpha-\gamma}{z-x}\cdot x+\beta+\frac{\alpha-\gamma}{z-x}\cdot y)  (\alpha+\frac{\alpha-\gamma}{z-x}\cdot x-(\beta+\frac{\alpha-\gamma}{z-x}\cdot y))} \\
&\quad - \frac{(\alpha+\gamma)^3}{(z+x)^4(\alpha -\frac{\alpha+\gamma}{x+z}\cdot x) (\alpha-\frac{\alpha+\gamma}{x+z}\cdot x+\beta-\frac{\alpha+\gamma}{x+z}\cdot y) (\alpha-\frac{\alpha+\gamma}{x+z}\cdot x-\beta+\frac{\alpha+\gamma}{x+z}\cdot y)}\\
&\quad + \frac{(\gamma-\beta)^3}{(z-y)^4 (\beta+\frac{\gamma-\beta}{y-z}\cdot y) (\beta+\frac{\gamma-\beta}{y-z}\cdot y+\alpha+\frac{\gamma-\beta}{y-z}\cdot x)  (\beta+\frac{\gamma-\beta}{y-z}\cdot y-(\alpha+\frac{\gamma-\beta}{y-z}\cdot x))}\\
&\quad - \frac{(\beta+\gamma)^3}{(z+y)^4 (\beta-\frac{\gamma+\beta}{y+z}\cdot y) (\beta-\frac{\gamma+\beta}{y+z}\cdot y-\alpha+\frac{\gamma+\beta}{y+z}\cdot x)  (\beta-\frac{\gamma+\beta}{y+z}\cdot y+\alpha-\frac{\gamma+\beta}{y+z}\cdot x)}\Bigg]\\
&= \frac{2\pi^4}{ 3}\frac{1}{(z^2-y^2)(z^2-x^2)}.
\end{align*}
Since the long expression is independent of the vector $v=(\alpha,\beta,\gamma)$, we can show the second equality by sending $\alpha$ to infinity. In this way, the third and fourth summand vanish.
\end{proof}

For $w=(0,0,1)$, we obtain the volume of the homogeneous K-contact structure we started with.
\begin{cor}
\[
\Vol(M,g_{(0,0,1)}) = \frac{2}{3}\pi^4.
\]
\end{cor}

\subsection{Homogeneous Sasaki manifolds: general case}

We consider a $(2n+1)$-dimensional compact homogeneous Sasaki manifold, which by \cite[Theorem 8.3.6]{BoyerGalicki} is the total space of a circle bundle of the form $\pi:G/K\to G/H$, where $H$ is the centralizer of a subtorus of the compact Lie group $G$. We can assume that the Reeb vector field $\xi$ of the homogeneous Sasaki structure on $G/K$ is the fundamental vector field of an element $b\in \mfg$; it then follows that $b$ is in the center of $G$. The $b$-orbits are exactly the $S^1$-fibers of the $S^1$-bundle $G/K\to G/H$, so $b$ is also contained in $\mfh$. We assume the normalization condition that the Reeb orbits of $G/K$ all have length $1$ (by applying a transverse homothety). Let $T\subset H$ be a maximal torus in $H$ (which is then also maximal in $G$), whose Lie algebra $\mft$ contains $b$. Then, $\mft$ splits as
\begin{equation}\label{eq:splitting}
\mft = (\mft\cap \mfk)\oplus \RR b.
\end{equation} Note that the example of the previous section fits into the notation of this section if we write $\SO(5)/\SO(3) = \SO(5)\times \SO(2)/ \SO(3)\times\SO(2)=G/K$, where the group $\SO(2)$ in the denominator is embedded diagonally.

Let $\xi'=b'^\#$, for $b'\in {\mathcal S}({\mathcal D},J)$, be the Reeb vector field of a deformed K-contact structure; for generic, i.e., dense choice of $b'$, the number of closed Reeb orbits is finite. In this case, they coincide with the one-dimensional orbits of the $T$-action on $G/K$ by left multiplication, which is the same as the $\pi$-preimage of the $T$-fixed point set in $G/H$. Denoting the contact form of the homogeneous K-contact structure by $\eta$, the deformed one with Reeb vector field $\xi'$ is given by
\[
\eta' = \frac{\eta}{\eta(\xi')}.
\]
As $\xi'$ does not have zeros on $G/K$, certainly $b'\notin \mfk$. Thus, the map $\mft\cap \mfk\to \mft\to \mft/\RR b' = \mfa$ is an isomorphism. Choose a vector $v\in \mft$.

We will apply our localization formula to show the following.
\begin{cor} \label{cor:volumehomogeneouscase} We have
\begin{multline}\label{eq:volumehomog}
\Vol(G/K,\eta')= \frac{\pi^n}{n!}\cdot  \sum_{w\in W(G)/W(H)} \\
\frac{1}{p(\Ad_{w^{-1}}(b'))^{n+1}}\cdot  \frac{p(\Ad_{w^{-1}}v)^n}{\prod_{\alpha\in \Delta_G\setminus\Delta_H} \alpha\left(\Ad_{w^{-1}}\left(v-\frac{p(\Ad_{w^{-1}}v)}{p(\Ad_{w^{-1}}b')}\cdot b'\right)\right)},
\end{multline}
where $W(G)$ (resp.\ $W(H)$) is the Weyl group of $G$ (resp.\ $H$),  $\Delta_G$ (resp.\ $\Delta_H$) is the root system of $G$ (resp.\ $H$) with respect to $\mft$ and $p:\mft\to \RR b\overset{b\mapsto 1}{\to} \RR$ denotes the projection along \eqref{eq:splitting}.
\end{cor}
As always, the right hand side of this formula is independent of $v$.
\begin{proof}
Let us calculate the summand in the localization formula \eqref{eq:volumelocalization} corresponding to the closed Reeb orbit through $wK$, $w\in W(G)$.  Then, along the Reeb orbit through $wK$, we have $b'=p(\Ad_{w^{-1}} b')b$. Thus, by our normalization condition for the original homogeneous Sasakian structure, the length of this closed Reeb orbit computes as
\[
\int_{T\cdot wK}\eta' = \frac{1}{p(\Ad_{w^{-1}}b')}.
\]

A similar calculation shows that the nominator in the respective summand is given by
\begin{equation}\label{eq:homgeneta}
\eta'(v^\#)|_{wK} =
\frac{p(\Ad_{w^{-1}}v)}{p(\Ad_{w^{-1}}b')}.
\end{equation}
To calculate the weights of the transverse $\mfa$-representation at $wK$, we first note that for each $w\in W(G)$, $\mfa$ is isomorphic to $\mft_{wK}$ via
\[
\mft_{wK}\longrightarrow \mft\longrightarrow \mft/\RR b'=\mfa.
\]
Moreover, the $\mft_{wK}$-slice representation on $\nu_{wK}T\cdot (wK)$ is equivalent to the $\mft_{wK}$-representation on $T_{wH}G/H$, whose weights are given by
\[
\{\left.\alpha\right|_{\mft\cap \mfk}\circ \Ad_{w^{-1}}\mid \alpha\in \Delta_G\setminus \Delta_H\}.
\]
Then, the set of transverse weights of the $\mfa$-representation is given by the set of all concatenations
\begin{equation}\label{eq:trvweightshomogeneous}
\xymatrix{\mfa \ar[r] & \mft_{wK} \ar[r]^{\Ad_{w^{-1}}} & \mft\cap \mfk \ar[r]^{\alpha} & \RR,}
\end{equation}
where $\alpha$ runs through $\Delta_G\setminus \Delta_H$. Concretely, we have that $v-\eta'(v^{\#})|_{wK}\cdot b'\in \mft_{wK}$; using \eqref{eq:homgeneta} we see that \eqref{eq:trvweightshomogeneous} is given by
\[
v+\RR b'\longmapsto \alpha\left(\Ad_{w^{-1}}\left(v-\frac{p(\Ad_{w^{-1}}v)}{p(\Ad_{w^{-1}}b')}\cdot b'\right)\right).
\]
In total, we obtain~\eqref{eq:volumehomog} by the localization formula \eqref{eq:volumelocalization}.
\end{proof}

\section{Localization of secondary characteristic numbers of Killing foliations}\label{sec:second}

\subsection{Secondary characteristic classes of Riemannian foliations}\label{sec:charclass}

Secondary characteristic classes for foliations are Chern-Simons invariants defined by the vanishing of certain characteristic forms of the normal bundle of foliations. See Section~\ref{sec:sch} for the introduction to secondary characteristic classes of foliations. We will show that certain secondary characteristic classes of Riemannian foliations can be localized to closed leaves by our localization formula (Theorem~\ref{thm:loc}).

Let us recall the construction of secondary characteristic classes of Riemannian foliations with framed normal bundle~\cite{LazarovPasternack1}. Let $M$ be a smooth manifold and $\F$ a Riemannian foliation of codimension $q$ on $M$ with trivial normal bundle. We fix a trivialization of $\nu \F$. Let $\nabla$ be the canonical basic Riemannian connection on $\nu\F$ (see~\cite[Theorem 5.9]{Tondeur}), which is determined by the transverse metric. Let $P$ be the orthonormal frame bundle of $\nu \F$. In terms of the connection form and the curvature of $\nabla$, we have the Chern-Weil homomorphism
\[
\widehat{\delta} : W(\so_{q}) \longrightarrow \Omega (P),
\]
where $W(\so_{q}) = \bigwedge \so_{q}^{*} \otimes S(\so_{q}^{*})$ is the Weil algebra of $\so_{q}$. Since $H(W(\so_{q})) = 0$, this map $\widehat{\delta}$ does not give any nontrivial information. By construction, the $i$-th Pontryagin form of $\widehat{\delta}$ is a basic form of degree $4i$. Therefore Pontryagin forms of degree greater than $q$ are trivial as proved by Pasternack. So $\widehat{\delta}(\oplus_{2k > q} S^{k}(\so_{q}^{*})) = 0$. Letting $W_{q}(\so_{q}) = W(\so_{q})/\oplus_{2k > q} S^{k}(\so_{q}^{*})$, we get the following induced map
\[
\delta : W_{q}(\so_{q}) \longrightarrow \Omega (P).
\]
Let $s : M \to P$ be the global section corresponding to the trivialization of $\nu \F$. Then the characteristic homomorphism $\Delta_{\F}$ of $(M,\F)$ is obtained by the composite
\[
\xymatrix{ W_{q}(\so_{q}) \ar[r]^<<<<<{\delta} & \Omega(P) \ar[r]^{s^{*}} & \Omega(M).}
\]
Elements in the image of the map induced in cohomology by $\Delta_{\F}$ are called {\em secondary characteristic classes} of $(M,\F)$.
The structure of $H(W_{q}(\so_{q}))$ was determined by Lazarov-Pasternack as follows: Let
\[
RW_q = \bigwedge (h_1,\ldots, h_{(q-1)/2}) \otimes \left( \RR[p_1, \ldots, p_{(q-1)/2}]/\{P \mid \deg P> q\} \right)
\]
 if $q$ is odd and
\[
RW_{q} = \bigwedge (h_1, \ldots, h_{(q-2)/2}, h_{e}) \otimes \left( \RR[p_1, \ldots, p_{(q-2)/2}, p_e]/\{P \mid \deg P> q\} \right)
\]
if $q$ is even. This $RW_{q}$ is a differential graded algebra (dga) with differential determined by $dh_{i} = p_{i}$, $dh_{e} =p_{e}$, $dp_{i}=0$ and $dp_{e}=0$. Then we have a natural isomorphism $H(W_{q}(\so_{q})) \cong H(RW_{q})$. Denoting $p_{i}(\nu \F, \F)=\Delta_{\F}(p_{i})$, $h_{i}(\F) = \Delta_{\F}(h_{i})$ and $h_{I}p_{J}(\F) = \Delta_{\F}(h_{I}p_{J})$,
\begin{itemize}
\item $p_{i}(\nu \F, \F)$ is the $i$-th Pontryagin form of $(\nu \F, \nabla)$,
\item $p_{e}(\nu \F, \F)$ is the Euler form of $(\nu \F, \nabla)$ and
\item $h_{i}(\F)$ is the Chern-Simons form $s^{*}T\delta(p_{i})$\cite[Remark 1.7]{LazarovPasternack1}.
\end{itemize}

\begin{rem}
Morita~\cite{Morita} defined more characteristic classes of Riemannian foliations by using the canonical affine Cartan connections on $P$ and the Weil algebra of affine groups.
\end{rem}

As already mentioned the main formula in \cite{Toeben} allowed the computation of the primary basic characteristic numbers through localization to the union of closed leaves. Now we will see how secondary characteristic numbers, i.e. the integrals of secondary characteristic classes of top degree, can be computed by localization.

\subsection{Application of the ABBV type formula to secondary characteristic numbers of Riemannian foliations}\label{subsec:loccharnumbers}

Let us apply our localization formula to compute secondary characteristic numbers of Killing foliations. We use the multi-index notation $h_{I}p_{J} := h_{i_{1}} \wedge \cdots \wedge h_{i_{k}} \wedge p_{j_{1}} \wedge \cdots \wedge p_{j_{m}}$ for $I = \{i_{1}, \ldots, i_{k}\}$ and $J = \{ j_{1}, \ldots, j_{m}\}$.

Let $M$ be a smooth manifold with a Riemannian foliation $\F$ of dimension $p$ and codimension $q$. Fix a trivialization of $\nu \F$. Then, we have a characteristic forms $h_{I}(\F)$ and $p_{J}(\nu \F, \F)$ of $(M,\F)$.
\begin{lemma}\label{lem:hI}
A form of type $h_I(\F)$ of degree $p$ is relatively closed.
\end{lemma}
\begin{proof}
Since $dh_{i_k}(\F)=p_{i_k}(\nu \F, \F)$ is basic and of degree larger than one, it follows that $dh_{i_{k}}(\F) \wedge h_{I -\{i_{k}\}}(\F)(v_1,\ldots,v_{p+1})=0$ whenever $p$ of the $p+1$ vectors $v_i$ is tangent to $\F$. Thus $h_{I}(\F)$ is relatively closed.
\end{proof}

In order to apply Theorem \ref{thm:loc}, we need to extend the basic form $p_J$ equivariantly. Assume that $\F$ is a Killing foliation. Let $\mfa$ denote the structural Killing algebra, which acts on $(M,\F)$ transversely (see Section~\ref{sec:tractiononKilling}). Let $\theta$ be the $\mfa$-invariant connection form of the basic Riemannian connection $\nabla$. As in the classical case of manifolds with Lie group actions, the equivariant curvature form
\[
F^{\theta}_{\mfa}:=d_{\mfa}\theta + \frac{1}{2}[\theta,\theta] \in \Omega^{2}_{\mfa}(M,\F) \otimes \operatorname{End}(\nu \F)
\]
can be used  to extend characteristic forms. By Chern-Weil construction, we get equivariant Pontryagin form $p_J(F^{\theta}_{\mfa})$, which is denoted by $p_{J,\mfa}(\nu \F, \F)$.
\begin{lemma}[{\cite[Section 7.2]{Toeben}}]\label{lem:equivariantpJ}
We have $p_{J,\mfa}(\nu \F, \F)(0) = p_{J}(\nu \F, \F)$ and $d_{\mfa}p_{J,\mfa}(\nu \F, \F)=0$. In other words, $p_{J,\mfa}(\nu \F, \F)$ is a natural equivariant extension of $p_{J}(\nu \F, \F)$ and represents a class in $H_\mfa(M,\F)$.
\end{lemma}
As a corollary of our ABBV-type localization formula we obtain the following result.
\begin{cor}\label{cor:locchar}
Let $M$ be a compact oriented manifold with an orientable Killing foliation $\F$ of dimension $p$ and codimension $q$. Assume that $\nu \F$ is trivial, and fix a trivialization to define secondary classes. If $\deg h_{I} = p$ and $\deg p_{J} = q$, then we have
\begin{equation}\label{eq:locchar}
\int_{M} h_Ip_J(\F)
=\sum_{k} l_{k}\int_{C_{k}/\F} \frac{i^*_{C_{k}}p_{J,\mfa}(\nu \F, \F)}{e_{\mfa}(\nu C_{k}, \F)},
\end{equation}
where $i_{C_{k}} : C_{k} \to M$ are the connected components of the submanifold of $M$ given by the union of closed leaves, $L_{k}$ is a leaf of $\F|C_{k}$ without holonomy and $l_{k}=\int_{L_{k}}h_{I}(\F)$. In particular, if $(M,\F)$ has only finitely many closed leaves $L_{1}$, $\ldots$, $L_{N}$, then we have
\begin{equation}\label{eq:locchar2}
\int_{M} h_Ip_J(\F)
=\sum_{k} \left( \int_{L_{k}} h_{I}(\F) \right) \int_{L_{k}} \frac{i^*_{L_{k}}p_{J,\mfa}(\nu \F, \F)}{i^{*}_{L_{k}}p_{e,\mfa}(\nu \F, \F)}
\end{equation}
and $\int_{M} h_Ip_{e}(\F) = \sum_{k} \int_{L_{k}} h_{I}(\F)$.
\end{cor}

\begin{proof}
By Lemma~\ref{lem:hI}, we have $\int_{M} h_Ip_J(\F) =\int_{(\F, h_I(\F))} p_J(\nu \F, \F)$. Then \eqref{eq:locchar} follows from Lemma~\ref{lem:equivariantpJ} and Theorem \ref{thm:loc}. If $(M,\F)$ has only finitely many closed leaves $L_{1}$, $\ldots$, $L_{N}$, then we have $\nu L_{k} = i^{*}_{L_{k}}\nu \F$ for $k=1$, $\ldots$, $N$. Because $e_\mfa(\nu L_{k}, \F)=i^*_{L_{k}}p_{e,\mfa}(\nu \F, \F)$, we obtain~\eqref{eq:locchar2} from  \eqref{eq:locchar}. The last equation follows from $\frac{i^*_{L_{k}}p_{e,\mfa}(\nu \F, \F)}{i^{*}_{L_{k}}p_{e,\mfa}(\nu \F, \F)}=1$.
\end{proof}

\subsection{The case of transversely K\"{a}hler flows}

Transversely K\"{a}hler foliations are defined in a similar way as for Riemannian foliations, and their secondary classes are richer than those of Riemannian foliations. A {\em transversely K\"{a}hler Haefliger cocycle} of complex codimension $m$ on a manifold $M$ is a quintuple $(\{U_i\},\{g_{i}\},\{\omega_i\},\{\pi_i\},\{\gamma_{ij}\})$ consisting of
\begin{enumerate}
\item an open covering $\{U_i\}$ of $M$,
\item submersions $\pi_i:U_i\to \RR^{2m}$,
\item Riemannian metrics $g_{i}$ on $\pi_{i}(U_{i})$,
\item symplectic form $\omega_{i}$ on $\pi_{i}(U_{i})$ such that $g_{i}$ and $\omega_{i}$ determine an integrable complex structure $J_i$ on $\pi_{i}(U_{i})$ by the equation $g_{i}(X,J_{i}Y) = \omega_{i}(X,Y)$,
\item transition maps $\gamma_{ij}:\pi_j(U_i\cap U_j)\to \pi_i(U_i\cap U_j)$ such that $\pi_i=\gamma_{ij}\circ \pi_j$, $\gamma_{ij}^{*}g_{i}=g_{j}$ and $\gamma_{ij}^{*}\omega_{i} = \omega_{j}$.
\end{enumerate}
The invariance of $\omega_{i}$ under transition maps implies that $\omega_{i}$ yields a $2$-form $\omega$ on $M$, which is called the {\em transverse K\"{a}hler form}. A transversely K\"{a}hler foliation is defined to be an equivalence class of transversely K\"{a}hler Haefliger cocycles, where the equivalence relation is defined by union as in the case of Riemannian foliations.

The equivalence relation of transversely K\"{a}hler Haefliger cocycles is defined as in the case of Riemannian foliations.

Two Riemannian Haefliger cocycles on $M$ are said to be equivalent if their union is a Riemannian Haefliger cocycle on $M$. A {\em Riemannian foliation} of codimension $q$ is defined to be an equivalence class of Riemannian Haefliger cocycles of codimension $q$.

Here we recall the characteristic classes of transversely K\"{a}hler foliations constructed by Matsuoka-Morita. Let $(M,\F)$ be a transversely K\"{a}hler foliation of complex codimension $m$ with framed normal bundle $\nu \F$. We have the eigenspace bundle decomposition $\nu_{\CC} \F = \nu^{1,0}\F \oplus \nu^{0,1} \F$ of the complex structure. Let $\pr : T_{\CC}M \to \nu_{\CC} \F$ be the projection and $E=\pr^{-1}(\nu^{0,1}\F)$. Take a bundle-like metric on $TM$ which induces the underlying transverse metric under the identification $(T\F)^{\perp} \cong \nu \F$. Define a connection $\nabla$ on $\nu^{1,0} \F$ by
\begin{equation}\label{eq:nabla}
\nabla_{X} Y=
\begin{cases}
\pr [X,\widetilde{Y}] & X\in C^\infty(T_\CC\F), Y\in C^{\infty}(\nu^{1,0})\\
\pr(\nabla^{LC}_{X}\widetilde{Y}) & X\in C^{\infty}(\nu_\CC\F), Y\in C^{\infty}(\nu^{1,0})
\end{cases}
\end{equation}
where $\pr(\widetilde{X}) = X$ and $\pr(\widetilde{Y}) = Y$. Since $E$ is involutive, $\nabla$ is well defined.

\begin{rem}
One can show that $\nabla_{X} Y= \pr [X,\widetilde{Y}]$ for $X \in C^\infty(E)$ and $Y\in C^{\infty}(\nu^{1,0})$ by using the fact that the $(0,1)$-part of the Levi-Civita connection of K\"{a}hler manifolds is equal to $\overline{\partial}$. Namely, $\nabla$ in~\eqref{eq:nabla} is a complex Bott connection in the sense of \cite[Definition 1.1.7]{Asuke3}.
\end{rem}

\begin{lemma}
The curvature form $F_{\nabla}$ of $\nabla$ is a basic $(1,1)$-form valued in $\operatorname{End}(\nu^{1,0}\F)$.
\end{lemma}

\begin{proof}
Here $\nabla$ obtained by complexifying the canonical basic connection in the sense of~\cite[p.~21]{Tondeur}. Thus, by the same computation as~\cite[Proof of Proposition 3.6]{Tondeur}, one sees that $F_{\nabla}$ is a basic $2$-form valued in $\operatorname{End}(\nu^{1,0}\F)$. The fact that $F_{\nabla}$ is of degree $(1,1)$ follows from a local argument. Take any element $\pi_i:U_i\to \RR^{2m}$ of the Haefliger cocycle. Here $\pi_{i}(U_{i})$ has a K\"{a}hler structure $(g_{i},\omega_{i})$. It is well known that the Levi-Civita connection $\nabla^{i}$ of $g_{i}$ on $T^{1,0}\pi_{i}(U_{i})$ is a Chern connection, which implies that the curvature form $F_{\nabla^{i}}$ is of type $(1,1)$ by well known argument. By the definition of $\nabla$, it is easy to see that we have $\pi_{i}^*\nabla_X Y=\nabla^{i}_{(\pi_{i})_{*}X}(\pi_{i})_{*}Y$ for basic $X$, $Y\in C^{\infty}(T_{\CC}M)$, namely, $\nabla$ is the pull back of $\nabla^{i}$ by $\pi_{i}$. Thus we have that the curvature form $F_{
 \nabla}$ of $\nabla$ is the pull back of that of $\nabla^{i}$. Then it follows that $F_{\nabla}|_{U_{i}}$ is of degree $(1,1)$.
\end{proof}

By the last lemma, the Chern forms of $\nabla$ of degree greater than $m$ are trivial. Thus the Chern-Weil homomorphism of $\nabla$ induces $\delta : H(KW_{m}) \longrightarrow H(U(\nu^{1,0}\F);\RR)$, where $U(\nu^{1,0}\F)$ is the unitary frame bundle of $\nu^{1,0}\F$ and
\[
KW_{m} = \bigwedge (u_1, \ldots, u_{m}) \otimes \left( \RR[s_1, \ldots, s_{m}, \Phi]/\{P \mid \deg P> m\} \right)
\]
is a dga with differential $d$ determined by $du_{i} = s_{i}$, $ds_{i} = 0$ and $d\Phi = 0$. Here $\Delta_{\F}(\Phi)$ is equal to the transverse K\"{a}hler form of $\F$. Then, as in the last section, the characteristic homomorphism
\[
\Delta_{\F} : H(KW_{m}) \longrightarrow H(M;\RR)
\]
is obtained by the composite of $\widehat{\delta}$ and $s^{*} : H(U(\nu^{1,0}\F);\RR) \to H(M;\RR)$, where $s$ is the section induced by the trivialization of $\nu^{1,0}(\F)$. 
We use the multi-index notation $u_{I}s_{J}$ as in the introduction and will denote $s_{J}(\nu \F, \F)=\Delta_{\F}(s_{J})$, $u_{I}(\F)=\Delta_{\F}(u_{I})$ and $ u_{I}s_{J}(\F)=\Delta_{\F}(u_{I}s_{J})$.

\begin{rem}
The original construction of Matsuoka-Morita of these classes was based on the canonical affine Cartan connection of $\nu^{1,0} \F$. They proved that it is equivalent to the construction above.
\end{rem}

Here $s_{1}(\nu^{1,0} \F) = \frac{\sqrt{-1}}{2\pi} \tr (F_{\nabla})$ is equal to the first Chern class of $\wedge^{m,0}\nu^{*} \F$. This fact allows us to define $u_{1}(\F)$ for a transversely K\"{a}hler foliation with trivialized $\wedge^{m,0}\nu^{*} \F$ but not necessarily framed normal bundle.

\begin{defn}\label{def:u1}
Let $(M,\F)$ be a manifold with a transversely K\"{a}hler foliation of complex codimension $m$ such that $\wedge^{m,0}\nu^{*} \F$ is trivial as a topological line bundle. Let $\varphi$ be a trivialization of $U(1)$-bundle $U(\wedge^{m,0}\nu^{*} \F) \to M$ associated to $\wedge^{m,0}\nu^{*} \F$. Then we define $u_{1}(\F)$ by
\[
u_{1}(\F) =  \frac{\sqrt{-1}}{2\pi} \varphi^{*} \theta,
\]
where $\theta \in \Omega^{1}(U(\wedge^{m,0}\nu^{*} \F)) \otimes \mathfrak{u}(1)$ is the connection form of the Chern connection on $\wedge^{m,0}\nu^{*} \F$ induced from $\nu^{1,0}\F$.
\end{defn}

\begin{rem}
The characteristic class $u_{1}s_{1}^{m}$ is called the Bott class. The Bott class of transversely holomorphic foliations is defined as a cohomology class of coefficient $\CC/\ZZ$ even for transversely holomorphic foliations $(M,\F)$ with nontrivial $\wedge^{m,0}\nu^{*} \F$ (see \cite{Asuke}).
\end{rem}

It is easy to see that we have $du_{1}(\F) = s_{1}(\nu \F, \F)$. Then, in the case where $\dim \F=1$, we see that $u_{1}(\F)$ is also relatively closed like as in Lemma~\ref{lem:hI}. Therefore we obtain the following corollary of Theorem~\ref{thm:loc}.

\begin{cor}\label{cor:trK2}
Let $(M,\F)$ be a compact manifold with an orientable taut transversely K\"{a}hler foliation of dimension one and complex codimension $m$. Assume that $\wedge^{m,0}\nu^{*} \F$ is trivial as a topological line bundle and $(M,\F)$ admits only finite closed leaves $L_{1}$, $\ldots$, $L_{N}$. For a given multiindex $J=(j_1,\ldots,j_l)$ with $j_1+\cdots+j_l= 2m$ we have
\begin{equation}\label{eq:u1sJloc}
\int_{M} u_{1} s_{J}(\F) = \sum_{k=1}^{N} \left(\int_{L_{k}} u_{1}(\F)\right) \frac{i^*_{L_{k}}s_{J,\mfa}(\nu \F, \F)}{i^{*}_{L_{k}}s_{m,\mfa}(\nu \F, \F)},
\end{equation}
where the $L_{k}$ are the isolated closed leaves of $\F$ and $s_{J,\mfa}(\nu \F, \F)$ is the equivariant characteristic form of $\F$ associated to $s_{J}$.
In particular, in the case where $J=\{m\}$, we obtain
\begin{equation}\label{eq:u1sm}
\int_{M} u_{1} s_{m}(\F) = \sum_{k} \int_{L_{k}}u_{1}(\F).
\end{equation}
\end{cor}

\begin{proof}
We first remark that an orientable taut transversely K\"{a}hler foliation of dimension one is a Killing foliation by~\cite[Th\'{e}or\`{e}me~A]{MolSer}. Since the equivariant Euler class of $\nu L_{k}$ is equal to $i^{*}_{L_{k}}s_{m,\mfa}(\nu \F, \F)$, Eq.~\eqref{eq:u1sJloc} is a direct consequence of Theorem~\ref{thm:loc}. Then Eq.~\eqref{eq:u1sm} follows from $\frac{i^{*}_{L_{k}}s_{m,\mfa}(\nu \F, \F)}{i^{*}_{L_{k}}s_{m,\mfa}(\nu \F, \F)} = 1$.
\end{proof}

Below we will apply Corollary~\ref{cor:trK2} to an example considered previously by Bott, Baum-Bott and Asuke. We will need the following well known computation of $u_{1}(\F)$.

\begin{lemma}\label{lem:u1}
Let $(M,\F)$ be a manifold with a transversely K\"{a}hler foliation of complex codimension $m$. Let $\varphi$ be a trivialization of $\wedge^{m,0}\nu^{*} \F$, which is regarded as an $m$-form on $M$. Assume that a $1$-form $\alpha$ on $M$ satisfies $d\varphi = 2\pi \sqrt{-1} \alpha \wedge \varphi$. If $\varphi$ is induced from a trivialization of $\nu^{1,0}\F$, then we have $\alpha(X) = u_{1}(\F)(X)$ for any tangential vector field $X\in C^\infty(T_\CC\F)$.
\end{lemma}

\begin{proof}
Since we will prove a local formula, we can assume that $\nu^{1,0}\F$ is trivialized with a trivialization $Z_{1}$, $\ldots$, $Z_{m}$ such that $\varphi(Z_{1}, \ldots, Z_{m})=1$. Let $\pr$ be the canonical projection $\pr : T_{\CC}M \to \nu^{1,0}\F$. Take $Y_{i}\in C^{\infty}(T_{\CC}M)$ so that $\pr Y_{i} = Z_{i}$. Recall that $E$ denotes the vector subbundle of $T_{\CC}M$ such that $T_{\CC}M/E = \nu^{1,0}\F$. For $X \in C^\infty(T_\CC\F)$, we have
\begin{align}
2\pi \sqrt{-1} \alpha(X) = & d\varphi (X, Y_{1}, \ldots, Y_{m}) \notag  \\
= & X \varphi (Y_{1}, \ldots, Y_{m}) + \sum_{i=1}^{m} (-1)^{i}  Y_{i} \varphi (X,Y_{1}, \ldots, \widehat{Y}_{i}, \ldots, Y_{m}) \notag \\
& + \sum_{1 \leq i < j \leq m} (-1)^{i+j}  \varphi ([Y_{i},Y_{j}], X, Y_{1}, \ldots, \widehat{Y}_{i}, \ldots, \widehat{Y}_{j}, \ldots, Y_{m}) \notag \\
& + \sum_{i=1}^{m} (-1)^{i} \varphi ([X,Y_{i}], Y_{1}, \ldots, \widehat{Y}_{i}, \ldots, Y_{m}) \label{eq:Bottclass}
\end{align}
Here the first three terms are zero. Since $\nabla_{X}Z_{i} = \pr[X,Y_{i}]$, we have
\[
\pr[X,Y_{i}] = \nabla_{X}Z_{i} = \sum_{j=1}^{m} \kappa_{ji}(X)Z_{j},
\]
where $\kappa$ is the connection form of $\nabla$ with respect to the framing $\{Z_i\}$. Thus it follows that the right hand side of~\eqref{eq:Bottclass} is equal to
\begin{multline*}
\sum_{i=1}^{m} (-1)^{i} \varphi ([X,Y_{i}], Y_{1}, \ldots, \widehat{Y}_{i}, \ldots, Y_{m}) =
\\ \sum_{i=1}^{m} (-1)^{i} \varphi (\kappa_{ii}(X)Y_{i}, Y_{1}, \ldots, \widehat{Y}_{i}, \ldots, Y_{m}) = - \sum_{i=1}^{m} \kappa_{ii}(X) = - \tr \kappa(X).
\end{multline*}
Since $u_{1}(\F) = \frac{\sqrt{-1}}{2\pi} \tr \kappa$ by~\cite[Eq. (3.5)]{ChernSimons}, one obtains that $\alpha(X) = u_{1}(\F)(X)$.
\end{proof}
Let $w=(w_0,\ldots,w_m) \in (\RR_{>0})^{m+1}$, and consider the deformation of type I of the standard Sasakian structure on $S^{2m+1}$ whose Reeb vector field $\xi_{w}$ is given by
\[
\xi_w = \sum_{i=0}^{m} w_i \left( x_i \frac{\partial}{\partial y_i} - y_i \frac{\partial}{\partial x_i} \right).
\]
Let $\F_{w}$ be the orbit foliation of the Reeb flow. Then $\F_{w}$ admits a transversely K\"{a}hler structure induced from the Sasakian structure. We choose $w$ generically so that the only closed leaves of $\F_w$ are those given by $|z_i|=1$, $i=0,\ldots,m$. We denote these by $L_0,\ldots,L_{m}$, respectively. We embed $S^{2m+1}$ to $\CC^{m+1}$ as a unit sphere. Then
\[
\sigma = \sqrt{-1}\sum_{j=0}^{m} (-1)^j w_jz_jdz_{0}\wedge \cdots \wedge \widehat{dz_j}\wedge \cdots \wedge dz_m
\]
is a nowhere vanishing global section of $\wedge^{m,0}\nu^{*} \F_{w}$. The trivialization of $\nu^{1,0} \F_{w}$ by $\sigma$ now gives to secondary classes. The following result is due to Bott \cite[p.\ 76]{Bott0}, Baum-Bott \cite[Example 1 in Section 11]{BaumBott} and Asuke \cite[Example 5.6]{Asuke3}. Note that in these references the authors consider transversely holomorphic foliations allowing $w$ to be complex.
The difference in techniques is as follows. The computation of Baum-Bott and Asuke is based on localization of the Chern classes $s_{1}s_{J}$ to the singular point of the foliation extended to $\CC^{n+1}$, while our computation is based on localization to the closed leaves of $\F$, which is an intrinsic process.

\begin{cor}\label{cor:AsukeExample}
We have
\begin{equation}
\int_{S^{2m+1}} u_{1}s_{J}(\F_{w}) = \frac{s_1s_J}{s_{m+1}}(w_0,\ldots,w_m).
\end{equation}
\end{cor}

\begin{proof}
By Lemma~\ref{lem:u1}, we have
\[
u_{1}(\F_{w})(X) = \frac{w_{0} + \ldots + w_{m}}{2\pi \sqrt{-1}} \left(\sum_j \frac{\overline{z}_{j} dz_{j}}{w_{j}} \right)(X)
\]
for $X \in C^{\infty}(TL_{k})$. Then we have
\begin{equation}
\int_{L_{k}} u_{1}(\F_{w}) = \frac{w_{0} +\ldots +  w_{m}}{w_{k}}.
\end{equation}
for $k=0,\ldots,m$.

We compute $i^*_{L_{k}}s_{J,\mfa}(\nu \F, \F)$ for $k=0$, $\ldots$, $m$. Consider the equivariant curvature form $F^{\theta}_{\mfa}=d_{\mfa}\theta + \frac{1}{2}[\theta,\theta] \in \Omega^{2}_{\mfa}(L_{k},\F) \otimes \operatorname{End}(\nu^{1,0} \F)$ of $i^*_{L_{k}}\nu^{1,0} \F$. Since $L_{k}$ is of dimension one, we have $F^{\theta}_{\mfa}(X) = - \iota_{X} \theta$ for $X \in \mfa$, where $\theta$ is the connection form of the Chern connection of $i^*_{L_{k}}\nu^{1,0} \F$. Letting $\{\alpha^k_j\}_{j=0, \ldots, \hat k, \ldots, m}$ be the weights of the transverse isotropy $\mfa$-representation at $L_k$, one can express $\iota_{X} \theta$ as a diagonal matrix whose diagonal entries are $\{\alpha_{j}^{k}(X)\}_{j=0, \ldots, m, j \neq k}$. It follows that
\[
i^*_{L_{k}}s_{J,\mfa}(\nu \F, \F)(X) = (-1)^{m}s_{J}((\alpha_{j}^{k}(X))_{j\neq k}).
\]
But we computed the transverse weights $\{\alpha^j_k\}$ of this example in Section \ref{subseq:Deformation}: Identifying $\mfa$ with $\mft/\RR b_{w}$, where $\mft$ is the standard torus acting on $\CC^{m+1}$ with standard basis $\{e_{j}\}$, for $v + \RR b_{w}=\sum_{i=0}^m \beta_je_j + \RR b_{w} \in \mfa$, it is $\alpha^k_j(v+\RR b_w) = \frac{\beta_k}{w_k} w_j - \beta_j$. Thus, by Corollary~\ref{cor:trK2}, we get
\begin{align*}
\int_{S^{2m+1}} u_1s_J(\F_{w}) &= (w_0+\cdots + w_m)\sum_{k=0}^m \frac{1}{w_k} \frac{s_J\left(\left(\frac{\beta_k}{w_k} w_j - \beta_j\right)_{j\neq k}\right)}{\prod_{j\neq k} \left(\frac{\beta_k}{w_k} w_j - \beta_j\right)}\\
&=  (w_0+\cdots + w_m)\sum_{k=0}^m \frac{1}{w_k} \frac{s_J(( w_j - w_k)_{j\neq k})}{\prod_{j\neq k} (w_j - w_k)}
\end{align*}
where in the second line we have set $\beta_j=1$, using the fact that the expression is independent of the chosen vector $v$. To finish the proof we have to show the following identity for elementary symmetric polynomials:
\begin{equation}
\sum_{k=0}^m \frac{1}{w_k} \frac{s_J(( w_j - w_k)_{j\neq k})}{\prod_{j\neq k} (w_j - w_k)} = \frac{s_{J}(w_{0},\ldots,w_m)}{\prod_j w_{j}}.
\end{equation}
By multiplying $\prod_j w_{j}$ to both sides, this is equivalent to
\begin{equation}\label{eq:w1}
\sum_{k=0}^m \frac{s_J(( w_j - w_k)_{j\neq k}) \prod_{j \neq k} w_{j} }{\prod_{j\neq k} (w_j - w_k)} = s_{J}(w_0,\ldots,w_m).
\end{equation}
Let $P(w_{0}, \ldots, w_{m})$ be the left hand side. First we prove that $P(w_{0}, \ldots, w_{m})$ is a polynomial in the variables $w_{j}$. By reducing the fraction to a common denominator, we have
\[
P(w_{0}, \ldots, w_{m}) = \frac{\sum_{k=0}^m  s_J(( w_j - w_k)_{j\neq k} )\left(\prod_{j \neq k} w_{j}\right) F_{k}(w_0,\ldots,w_m)}{\prod_{i < h} (w_{i} - w_{h})},
\]
where $F_{k}(w_0,\ldots,w_m)=(-1)^{m-k}\prod_{i < h,\, i,h \neq k} (w_{i} - w_{h})$. Note that for $s\neq t$ we have $F_t = (-1)^{t-s-1} F_s$ on the subspace $w_s=w_t$. The nominator
\begin{equation}\label{eq:strangenominator}
\sum_{k=0}^m s_J(( w_j - w_k)_{j\neq k}) \left(\prod_{j \neq k} w_{j}\right) F_{k}(w_0,\ldots,w_m)
\end{equation}
is divisible by $w_{s} - w_{t}$ for any $s < t$: Indeed, if we substitute $w_{s} = w_{t}$ in \eqref{eq:strangenominator}, only the $s$-th and $t$-th summands are nonzero, and they cancel each other due to $F_t=(-1)^{t-s-1}F_s$. This shows that $P(w_{0}, \ldots, w_{m})$ is a polynomial in the variables $w_{j}$, and it is  obviously homogeneous of degree $2|J|=m$.

We see that $P(w_{0}, \ldots, w_{m}) = s_{J}(w_{0},\ldots,w_m)$ holds when at least one of the $w_{j}$ is zero. Indeed, if $w_{s}=0$, then only the $s$-th summand of $P(w_{0}, \cdots, w_{m})$ is nonzero, and of the form $s_{J}((w_j)_{j\neq s})=s_J(w_0,\ldots,w_{s-1},0,w_{s+1},\ldots,w_m)$, which is exactly the right hand side of Equation \eqref{eq:w1}. It follows that all the monomials on the left and right hand side of Equation \eqref{eq:w1} which do not involve all the variables $w_j$ coincide. But as both sides are polynomials of degree $m$ in $m+1$ variables, all monomials are of this form.
\end{proof}

\appendix
\section{The equivariant basic Thom homomorphism}
\subsection{Equivariant basic Euler classes}\label{sec:eqcharclasses}
Let us recall the definition of the equivariant basic Euler class. It is a special case of equivariant basic characteristic classes (see \cite[Section 7.1]{Toeben}).
Consider a foliated $\SO(r)$-bundle $(P,\E)\to (N,\F)$ (see \cite[Section 2]{Toeben}), with a transverse $\mfh$-action on total and base space that commutes on the former with the $\SO(r)$-action, and with an $\mfh$-invariant connection $\theta\in \Omega^1(P,\E)\otimes \so(r)$.

\begin{defn}\label{defn:eqbasicEulerform}
The \emph{equivariant basic Euler form} $e_\mfh(P,\E)$ of $(P,\E)\to (N,\F)$ is defined by
\begin{equation}\label{eq:eqbasiceuler}
e_\mfh(P,\E)(X) = \frac{1}{(2\pi)^{r/2}}{\det}^{1/2}(F^\theta_\mfh(X)) = \frac{1}{(2\pi)^{r/2}}{\det}^{1/2}(F^\theta - \iota_X\theta),
\end{equation}
where $F^\theta_\mfh = d_\mfh \theta + \frac{1}{2}[\theta,\theta]\in \Omega^2_{\mfh}(P,\E)\otimes \so(r)$ is the equivariant curvature form and $\frac{1}{(2\pi)^{r/2}}{\det}^{1/2}\in S(\so(r)^*)^{\so(r)}$ is the Pfaffian, where $\det^{1/2}$ of a $2n\times 2n$ block diagonal matrix whose $i$th block matrix is $\left(\begin{smallmatrix} 0 & -\lambda_i \\ \lambda_i & 0 \end{smallmatrix}\right)$ is defined to be $(-1)^n\lambda_1\cdots\lambda_n$. The connection form $\theta$, on which the Euler form depends, is suppressed in the notation for simplicity.
\end{defn}

For our applications the most important special case is the equivariant basic Euler form of the normal bundle of a connected component $N$ of the union of closed leaves $C$ of a Killing foliation $\F$. We assume $\nu N$ to be oriented. Let $P$ be the corresponding principal foliated bundle of its oriented orthonormal frames. It carries a Riemannian foliation $\E$, where a leaf through a given orthonormal frame $\xi$ of $\nu N$ with foot point $p$ consists of all those frames obtained by sliding $\xi$ along the leaf through $p$. Then the projection $(P,\E)\to (N,\F|N)$ is a foliated principle bundle. Since the $\mfa$-action of $(M,\F)$ is trivial on $N$ and preserves each normal space $\nu_p N$, the foliated bundle is equivariant with respect to the restricted $\mfa$-action. By $e_\mfa(\nu N,\F):=e_\mfa(P,\E)$ we denote the equivariant basic Euler form of this foliated bundle.
Let us explain that in this case, similarly to the ordinary equivariant Euler form, it is just the product of the weights of the (transverse) $\mfa$-isotropy representation. The weights appear in the following way. As $\mfa$ acts transversely isometrically (\cite{GT2010}), we have the representation $\mfa\to \mathfrak{so}(\nu_pN);\, X\mapsto [X^\flat,\cdot]$, $p\in N$ arbitrary, where $X^{\flat}$ denotes the associated transverse vector field for $X\in \mfa$. Since $\mfa$ is abelian, with respect to some oriented orthonormal frame $\xi:\RR^k\to \nu_pN$, every $X\in\mfa$ is a block diagonal matrix with blocks of the form $\left(\begin{smallmatrix} 0 & -\alpha_j(X) \\ \alpha_j(X) & 0 \end{smallmatrix}\right)$ which defines the weights $\alpha_j\in \mfa^*$ and weight subbundles $V_j$ of $\nu N$.
\begin{prop}\label{prop:Eulerclassweights}
Let $\F$ be a Killing foliation of codimension $k$, and let $N$ be a connected component of the union of closed leaves, with oriented normal bundle $\nu N$. Then $\nu N$ can be decomposed as
\begin{equation}\label{eq:nuCdectoric}
\nu N = \bigoplus_{j}V_j\;,
\end{equation}
where $V_j$ is an $\mfa$-invariant vector bundles of rank $2$ whose weight is given by $\alpha_j$. Then
\begin{equation}\label{eq:weights}
e_{\mfa}(\nu N, \F)(X) = \frac{1}{(-2\pi)^{k/2}}\prod_{j} \alpha_j(X) + \mbox{(terms of lower polynomial degree)}.
\end{equation}
If, in particular, $N$ consists of only one leaf, then the terms of lower polynomial degree vanish.
\end{prop}

\begin{proof}
The purely polynomial part of the equivariant basic Euler form $e_\mfa(\nu N,\F)\in H_\mfa(N,\F)=H(N,\F)\otimes S(\mfa^*)$ is the equivariant basic Euler form $e_\mfa(\nu N|L,\F)\in H_\mfa(L,\F)=S(\mfa^*)$ of its restriction to a generic leaf $L$ of $\F|N$.
Let $\xi:\RR^k\to\nu_pN$ be an adapted orthonormal frame of the normal space as before and denote by $\theta$ the $\so(k)$-connection associated to the Levi-Civita connection of the transverse metric on $\nu N$. We consider the induced transverse action of $\mfa$ on the bundle of orthonormal frames of $\nu N$. We claim that $\theta_\xi(-X^\flat_\xi)$, $X\in\mfa$, is the same block diagonal matrix described above with blocks of the form $\left(\begin{smallmatrix} 0 & -\alpha_j(X) \\ \alpha_j(X) & 0 \end{smallmatrix}\right)$. Below, we consider the flow on local leaf spaces generated by the transverse vector field $X^\flat$.  Now observe
\begin{align*}
\{\RR^k\ni v\longmapsto [X^\flat,\xi(v)](p)\} &=\left\{v\longmapsto \left.\frac{d}{dt}\right|_{t=0} d\exp(-tX^\flat)\xi(v)\right\} \\ &= \left.\frac{d}{dt}\right|_{t=0} (d\exp (-tX^\flat)\circ\xi) =-X^\flat_\xi =\xi\circ\theta_\xi(-X^\flat_\xi).
\end{align*}
In the second equality we pass from the isometric action on the local leaf space to the induced action on the orthonormal frame bundle of the tangent bundle of the local leaf space; we see that $\exp(-tX^\flat)\circ\xi$ is a vertical curve in the frame bundle with tangent vector $-X^\flat_\xi$ in $\xi$. This proves our claim. Since $\Omega(L,\F)$ is trivial for degrees greater than $0$, the restriction of $F^\theta$ to $L$ vanishes and we have
\begin{align*}
e_{\mfa}(\nu N|L, \F)(X) &= \frac{1}{(2\pi)^{k/2}}{\det}^{1/2}(-\iota_{X^{\flat}} \theta)  \\ &= \frac{1}{(2\pi)^{k/2}}\prod_{j} {\det}^{1/2}\left(\begin{smallmatrix} 0 & -\alpha_j(X) \\ \alpha_j(X) & 0 \end{smallmatrix}\right) = \frac{1}{(-2\pi)^{k/2}}\prod_{j} \alpha_j(X).
\end{align*}

\end{proof}

\subsection{The basic Thom form}\label{sec:universalthomform}
In this section we define equivariant basic Thom classes by adapting the construction of Mathai and Quillen \cite{MathaiQuillen} of a universal Thom form to our basic setting. We follow \cite[Sections 10.2 and 10.3]{GS}. The basic steps of this construction have also been indicated in \cite[Section 6]{Toeben}.

Assume that $\F$ is Killing, so that we have the natural transverse action of the structural Killing algebra $\mfa$. Let $i:N\to M$ be a closed stratum of the $\mfa$-action, i.e., $N$ is a connected component of the union $C$ of closed leaves of $\F$. Let $r$ be the codimension of $N$. Denote by $p:\nu N\to N$ the normal bundle of $N$, which we identify with a small saturated tubular neighborhood around $N$ by the normal exponential map $\nu N \to M$. We denote the induced foliation on $\nu N$ by $\whF$. Let $P$ be the $\SO(r)$-bundle of oriented frames in $\nu N$. On $P$ we have a natural foliation $\E$; a leaf of $\E$ through a given frame $\xi$ consists of all frames that are obtained by sliding $\xi$ along the leaf of $\F$ containing the foot point of $\xi$. Now the projection $\pi:P\times \RR^r\to P\times_{\SO(r)} \RR^r = \nu N$ sends the product foliation $\E\times \{*\}$ to $\whF$.

Let $\mu\in \Omega_{\so(r)}(\RR^r)$ be the universal Thom form of Mathai and Quillen given by
\[
\mu(X) = \frac{e^{-\|x\|^2}}{\pi^{r/2}} \sum_I \varepsilon_I {\det}^{1/2}\left(\frac{X_I}{2}\right)dx_{I^c},
\]
see also \cite[Eq.~(40)]{Meinrenken}. Here, the sum is taken over all subsets $I$ of $\{1,\ldots r\}$ with an even number of elements. For $X\in \so(r)$, $X_I$ denotes the matrix obtained from $X$ by deleting those rows and columns that are not in $I$, and $I^c$ denotes the complement of $I$. The sign $\varepsilon_I\in \{\pm 1\}$ is defined by $dx_{I} dx_{I^c} = \varepsilon_I dx_1\cdots dx_r$. We let $\rho$ be the $\SO(r)$-equivariant diffeomorphism from the open unit ball $B\subset \RR^r$ to $\RR^r$ given by
\[
\rho(x) = \frac{x}{1-\|x\|^2}\;.
\]
Then $\rho^*\mu$ can be extended to an equivariant differential form of compact support on $\RR^r$ by setting it to zero outside of $B$. We consider now the form $\rho^*\mu\otimes 1\in \Omega_{\mfa\times \so(r),cv}(\RR^r)$, where $\mfa$ is supposed to act trivially on $\RR^r$.

Note that the connection form $\theta$ on $P$ of the canonical basic Riemannian connection is $\mfa$-invariant. The Cartan map with respect to $\theta$ is defined by
\[\begin{array}{cccc}
\kappa_\mfa : & \Omega_{\mfa\times \so(r),cv}(P\times \RR^r,\E\times\{*\})& \longrightarrow & \Omega_{\mfa,cv}(\nu N,\whF); \\
   & f\otimes \omega & \longmapsto & f(F^\theta_\mfa)\wedge \omega_{\hor \so(r)},
\end{array}\]
where $F^\theta_\mfa = d_\mfa \theta + \frac12 [\theta,\theta]\in \Omega^2_\mfa(P,\E)\otimes \so(r)$ is the $\mfa$-equivariant curvature of $\theta$, and the subscript $\hor \so(r)$ denotes projection on the horizontal component (for the projection see e.g.~\cite[p.~58]{GS}). Moreover, we consider here equivariant differential forms with compact vertical support (on the left hand side with respect to the vector bundle projection $P\times \RR^r\to P$). These do not form a differential graded algebra of type (C), as the connection form does not necessarily have compact support, but still they are a $W^*$-module in the sense of \cite[Section 3.4]{GS}.

Let ${\mathrm{pr}}_2:P\times \RR^r\to \RR^r$ be the second projection.

\begin{defn}
The {\em equivariant basic Thom class} $\Phi_{\mfa}$ of $(N,\F)$ in $(M,\F)$ is defined as the cohomology class of the image of $\rho^*\mu \otimes 1$ under  the composition
\[
\Omega_{\mfa\times \so(r),cv}(\RR^r) \overset{{\mathrm{pr}}_2^*}{\longrightarrow} \Omega_{\mfa\times \so(r),cv}(P\times \RR^r,\E\times\{*\})\overset{\kappa_\mfa}\longrightarrow \Omega_{\mfa,cv}(\nu N,\whF),
\]
where the compact vertical support in the middle refers to the projection $P\times \RR^r\to P$. So
\begin{equation}\label{eq:thomform}
\Phi_\mfa := \kappa_\mfa({\mathrm{pr}}_2^*(\rho^*\mu \otimes 1)).
\end{equation}
\end{defn}

We state two fundamental properties of equivariant basic Thom classes. The map $p:\nu N\to N$ as a foliated map is equivariant with respect to the transverse action of the structural Killing algebra $\mfa$. Therefore the fiber integration $p_*:\Omega^{r+\bullet}_{cv}(\nu N,\whF)\to\Omega^{\bullet}(N,\F)$ is an $\mfa$-dga map. Thus we obtain an $S(\mfa^*)$-homomorphism $p_*:H^{r+\bullet}_{cv,\mfa}(\nu N,\whF)\to H^\bullet_\mfa(N,\F)$, which is an isomorphism (see \cite[p.~20]{Toeben}).

\begin{lemma}
We have
\begin{align}
p_*\Phi_\mfa & =1, \label{eq:thomisthom}\\
j^*\Phi_\mfa & = e_\mfa(\nu N,\F), \label{eq:thomeuler}
\end{align}
where $j:N \to \nu N$ is the canonical inclusion and $e_\mfa(\nu N,\F)$ is the equivariant basic Euler class of the bundle $P\to N$ (see Eq.~\eqref{eq:eqbasiceuler}).
\end{lemma}
For the proof, we refer to \cite[Section 10.4]{GS} and \cite[Section 10.5]{GS}, respectively (see also \cite{Meinrenken}). As $p_*$ is an isomorphism as mentioned above, \eqref{eq:thomisthom} characterizes the equivariant basic cohomology class $[\Phi_\mfa]$.

\subsection{The equivariant basic Thom homomorphism}

The following is a well known result in the classical context, which is necessary to define the equivariant basic Thom homomorphism. Recall that $N$ is a connected component of the union $C$ of closed leaves.
\begin{lemma}\label{lem:A5}
A map $\tau_\mfa: \Omega^{\bullet}_\mfa(N,\F) \longrightarrow \Omega^{r+\bullet}_{\mfa,cv}(\nu N,\whF)$ defined by
$$
\tau_\mfa(\omega)=p^*\omega\wedge \Phi_\mfa,
$$
where $\Phi_\mfa\in \Omega_{\mfa,cv}^r(\nu N,\whF)$ is the basic Thom form (see \eqref{eq:thomform}), induces the inverse of $p_*:H^{r+\bullet}_{\mfa,cv}(\nu N,\whF)\to H^\bullet_\mfa(N,\F)$.
\end{lemma}

\begin{proof}
For each equivariant basic form $\omega$ we can verify
\[
(p_*\circ\tau_\mfa(\omega))(X)=p_*(p^*\omega(X)\wedge \Phi_\mfa(X))=\omega(X)\wedge p_*(\Phi_\mfa(X))=\omega(X)
\]
for all $X\in \mfa^*$ by using the projection formula for forms \cite[Prop.~IX, I.7.13]{GHV} and Equation \eqref{eq:thomisthom}.
\end{proof}

\begin{defn}\label{defn:eqThomhom}
The composition
\[
\xymatrix{H^\bullet_\mfa(N,\F) \ar[r]^<<<<<{\tau_\mfa} & H^{r+\bullet}_{\mfa,cv}(\nu N,\whF) \ar[r] & H^{r+\bullet}_\mfa(M,\F)}
\]
is denoted by $i_{*}$ and called the {\em equivariant basic Thom homomorphism}, where the second map is induced by the inclusion $(\nu N,\whF) \to (M,\F)$.
\end{defn}

\end{document}